\documentclass[11pt,american,preprint]{elsarticle}
\makeatletter
\def\ps@pprintTitle{%
 \let\@oddhead\@empty
 \let\@evenhead\@empty
 \def\@oddfoot{\centerline{\thepage}}%
 \let\@evenfoot\@oddfoot}
\makeatother

\usepackage[margin=1in]{geometry}

\usepackage{xcolor}
\usepackage{lipsum}
\usepackage{changepage}
\usepackage{centernot}

\usepackage[T1]{fontenc}
\usepackage[latin9]{luainputenc}
\usepackage{amsmath}
\usepackage{amsthm}
\usepackage{amssymb}
\usepackage{stmaryrd}
\usepackage{esint}
\usepackage{nameref}
\usepackage{hyperref} 
\usepackage{cleveref} 

\usepackage[shortlabels]{enumitem}
\usepackage{chngcntr}

\usepackage{mathrsfs}

\newsavebox{\foobox}

\usepackage{graphicx,amssymb}

\usepackage{array}
\newcolumntype{M}[1]{>{\centering\arraybackslash}m{#1}}

\usepackage{bbm} 
\makeatletter
\numberwithin{equation}{section}
\theoremstyle{plain}
\newtheorem{thm}{\protect\theoremname}[section]

\theoremstyle{plain*}
\newtheorem*{thm*}{\protect\theoremname}

\theoremstyle{plain}
\newtheorem{lem}[thm]{\protect\lemmaname}
  
\theoremstyle{plain*}
\newtheorem*{lem*}{\protect\lemmaname}  
  
  \theoremstyle{plain}
  \newtheorem{prop}[thm]{\protect\propositionname}
  
    \theoremstyle{plain*}
  \newtheorem*{prop*}{\protect\propositionname}

\theoremstyle{remark}

\theoremstyle{remark*}
\newtheorem*{question*}{Question} 
\theoremstyle{remark}
\newtheorem{rem}[thm]{\protect\remarkname}
\theoremstyle{remark*}
\newtheorem*{rem*}{\protect\remarkname}
\theoremstyle{remark}

\theoremstyle{remark*}
\newtheorem*{example*}{\protect\examplename}

\theoremstyle{plain}
\newtheorem{cor}[thm]{\protect\corollaryname}
\providecommand{\corollaryname}{Corollary}
  
\theoremstyle{definition}
  
\makeatother

\theoremstyle{plain} 
\newcommand{\thistheoremname}{}
\newtheorem{genericthm}[thm]{\thistheoremname}

\newtheorem*{genericthm*}{\thistheoremname}
\newenvironment{namedthm*}[1]
  {\renewcommand{\thistheoremname}{#1}%
   \begin{genericthm*}}
  {\end{genericthm*}}

\makeatother

\usepackage{babel}
 \providecommand{\lemmaname}{Lemma}
  \providecommand{\propositionname}{Proposition}
  \providecommand{\remarkname}{Remark}
\providecommand{\theoremname}{Theorem}

\newcommand{\R}{\mathbb{R}}
\newcommand{\N}{\mathbb{N}}
\newcommand{\Q}{\mathbb{Q}}
\newcommand{\Z}{\mathbb{Z}}

\newcommand\precdot{\mathrel{\ooalign{$\prec$\cr
  \hidewidth\raise0ex\hbox{$\cdot\mkern0.5mu$}\cr}}}
\newcommand\preceqdot{\mathrel{\ooalign{$\preceq$\cr
  \hidewidth\raise0.225ex\hbox{$\cdot\mkern0.5mu$}\cr}}}

\usepackage{fancyhdr}

\usepackage{chngcntr}
\usepackage{apptools}
\AtAppendix{\counterwithin{thm}{section}}

\begin{document}
\begin{frontmatter}
\title{Mixing and rigidity along asymptotically linearly independent sequences}

\author[add1]{Rigoberto Zelada}
\ead{zeladacifuentes.1@osu.edu}

\address[add1]{Department of Mathematics. The Ohio State University, Columbus, OH 43210, USA}

\begin{abstract}
We utilize Gaussian measure preserving systems to prove the existence and genericity of Lebesgue measure preserving transformations $T:[0,1]\rightarrow [0,1]$ which exhibit both mixing and rigidity behavior along families of \emph{asymptotically linearly independent} sequences. 
Let  $\lambda_1,...,\lambda_N\in[0,1]$ and let $\phi_1,...,\phi_N:\mathbb N\rightarrow\mathbb Z$ be asymptotically linearly independent (i.e. for any $(a_1,...,a_N)\in\mathbb Z^N\setminus\{\vec 0\}$, $\lim_{k\rightarrow\infty}|\sum_{j=1}^Na_j\phi_j(k)|=\infty$). Then the class of invertible Lebesgue measure preserving transformations $T:[0,1]\rightarrow[0,1]$ for which there exists a sequence $(n_k)_{k\in\N}$ in $\N$ with 
$$\lim_{k\rightarrow\infty}\mu(A\cap T^{-\phi_j(n_k) }B)= (1-\lambda_j)\mu(A\cap B)+\lambda_j\mu(A)\mu(B),$$ 
for any measurable $A,B\subseteq [0,1]$ and any $j\in\{1,...,N\}$, is generic.\\
This result is a refinement of a result due  to A. M. St{\"e}pin \cite[Theorem 2]{stepin1987spectral} and a generalization of a result due to V. Bergelson, S. Kasjan, and M. Lema{\'n}czyk \cite[Corollary F]{BKLUltrafilterPoly}.
\end{abstract}
\begin{keyword}
Ergodic theory, Gaussian systems, generic transformation, rigidity sequence.  
\end{keyword}
\end{frontmatter}
\tableofcontents
\section{Introduction}\label{Section1Gaussian}
Let $([0,1],\mathcal B,\mu)$ be the  probability space
where $\mathcal B=\text{Borel}([0,1])$ and $\mu$ is the Lebesgue measure. Denote 
by  $\text{Aut}([0,1],\mathcal B,\mu)$  the set of invertible measure preserving transformations $T:[0,1]\rightarrow [0,1]$ endowed with the weak topology (i.e. the topology defined on $\text{Aut}([0,1],\mathcal B,\mu)$ by $T_n\rightarrow T$ if and only if for each $f\in L^2(\mu)$, $\|T_nf-Tf\|_{L^2}\rightarrow 0$). With this topology $\text{Aut}([0,1],\mathcal B,\mu)$ is a completely metrisable space.\\

St{\"e}pin proved in \cite[Theorem 2]{stepin1987spectral} that, given $\lambda\in[0,1]$, the set of transformations $T\in \text{Aut}([0,1],\mathcal B,\mu)$ for which there exists an increasing sequence $(n_k)_{k\in\N}$ in $\N=\{1,2,...\}$ such that for any $A,B\in\mathcal B$,
\begin{equation}\label{0.StepinResult}
\lim_{k\rightarrow\infty}\mu(A\cap T^{-n_k}B)=(1-\lambda)\mu(A\cap B)+\lambda\mu(A)\mu(B),
\end{equation}
is a dense $G_\delta$ set in $\text{Aut}([0,1],\mathcal B,\mu)$. A refinement of St{\"e}pin's theorem, which is a special case of  \cref{0.IndependentPolysStepinResult} below, states that for any (strictly) monotone sequence $\phi:\N\rightarrow\Z$ and any $\lambda\in [0,1]$,
the set $\mathcal G(\phi,\lambda)$ consisting of all transformations $T\in \text{Aut}([0,1],\mathcal B,\mu)$ for which there exists an increasing sequence $(n_k)_{k\in\N}$ in $\N$ such that for any $A,B\in\mathcal B$, 
$$\lim_{k\rightarrow\infty}\mu(A\cap T^{-\phi(n_k)}B)=(1-\lambda)\mu(A\cap B)+\lambda\mu(A)\mu(B),$$
is again dense $G_\delta$.\\

It follows that for any $\lambda_1,\lambda_2\in[0,1]$ and any monotone sequences $\phi_1,\phi_2:\N\rightarrow\Z$, the set $\mathcal G(\phi_1,\lambda_1)\cap \mathcal G(\phi_2,\lambda_2)$ is residual (i.e. it contains a dense $G_\delta$ set). Thus, there exists $T\in\text{Aut}([0,1],\mathcal B,\mu)$ such that for some increasing sequences $(n^{(1)}_k)_{k\in\N}$ and $(n^{(2)}_k)_{k\in\N}$ in $\N$ and any $A,B\in\mathcal B$,
\begin{equation}\label{0.FirstExpresion}
\lim_{k\rightarrow\infty}\mu(A\cap T^{-\phi_1(n^{(1)}_k)}B)=(1-\lambda_1)\mu(A\cap B)+\lambda_1\mu(A)\mu(B)
\end{equation}
and
\begin{equation}\label{0.SecondExpression}
\lim_{k\rightarrow\infty}\mu(A\cap T^{-\phi_2(n^{(2)}_k)}B)=(1-\lambda_2)\mu(A\cap B)+\lambda_2\mu(A)\mu(B).
\end{equation}
Note that depending on our choice of $\lambda_1$, $\lambda_2$, $\phi_1$, and $\phi_2$, it might be the case that for every $T\in\mathcal G(\phi_1,\lambda_1)\cap \mathcal G(\phi_2,\lambda_2)$, the sequences $(n^{(1)}_k)_{k\in\N}$ and $(n^{(2)}_k)_{k\in\N}$ in \eqref{0.FirstExpresion} and \eqref{0.SecondExpression} must be different.\\

For instance, when $\lambda_1=0$, $\lambda_2=1$, $\phi_1(n)=n$, and $\phi_2(n)=2n$ for each $n\in\N$, we have that if \eqref{0.FirstExpresion} and \eqref{0.SecondExpression} hold for some $T\in\mathcal G(\phi_1,\lambda_1)\cap \mathcal G(\phi_2,\lambda_2)$, then 
$$\lim_{k\rightarrow\infty}|n^{(1)}_k-n^{(2)}_k|=\infty.$$
To see this, suppose for sake of contradiction that $\lim_{j\rightarrow\infty}n^{(1)}_{k_j}-n^{(2)}_{k_j}=a\in\Z$ for some increasing sequence $(k_j)_{j\in\N}$ in $\N$. Since $\lambda_1=0$,
$$\lim_{j\rightarrow\infty}\mu(A\cap T^{-2n^{(1)}_{k_j}}B)=\mu(A\cap B),$$
for any $A,B\in\mathcal B$. 
Picking $A\in\mathcal B$ with $\mu(A)\in (0,1)$ and letting $B=T^{-2a}A$, we obtain
\begin{equation*}\label{0.BadLimit}
\mu^2(A)=\mu(A)\mu(B)=\lim_{j\rightarrow\infty}\mu(A\cap T^{-2n^{(2)}_{k_j}}B)
=\lim_{j\rightarrow\infty}\mu(A\cap T^{-2n^{(1)}_{k_j}+2a}B)=\mu(A\cap T^{2a}B)=\mu(A).
\end{equation*}
Noting that $\mu^2(A)\neq \mu(A)$, we reach the desired contradiction.\\ 

The following result, which is  a consequence of \cite[Corollary F]{BKLUltrafilterPoly}, provides sufficient conditions on sequences of the form 
$(v_1(k))_{k\in\N}$ and $(v_2(k))_{k\in\N}$, where $v_1,v_2 \in\Z[x]$, to ensure the existence of a $T\in \text{Aut}([0,1],\mathcal B,\mu)$ such that \eqref{0.FirstExpresion} and \eqref{0.SecondExpression} hold with $(n^{(1)}_k)_{k\in\N}=(n^{(2)}_k)_{k\in\N}$ and arbitrary $\lambda_1,\lambda_2\in\{0,1\}$. We denote the set of all (strictly) increasing sequences $(n_k)_{k\in\N}$ in $\N$ by $\N^\N_\infty$.
\begin{thm}\label{0.IndependentPolysBKL}
Let $N\in\N$, let $\lambda_1,...,\lambda_N\in\{0,1\}$, and let  $v_1,...,v_N\in\Z[x]$ be  $\Q$-linearly independent polynomials such that $v_j(0)=0$ for each $j\in\{1,...,N\}$. Then the set
\begin{multline*}
    \{T\in\text{Aut}([0,1],\mathcal B,\mu)\,|\,\exists (n_k)_{k\in\N}\in\N^\N_\infty\,\forall j\in\{1,...,N\}\,\forall A,B\in\mathcal B,\\
    \lim_{k\rightarrow\infty}\mu(A\cap T^{-v_j(n_k)}B)=(1-\lambda_j)\mu(A\cap B)+\lambda_j\mu(A)\mu(B)\}
\end{multline*}
is a dense $G_\delta$ set.
\end{thm}
\cref{0.IndependentPolysStepinResult} below, which we prove in Section \ref{Section6InterpolationResults}, extends \cref{0.IndependentPolysBKL} to any real numbers $\lambda_1,...,\lambda_N\in[0,1]$ and arbitrary \textbf{asymptotically linearly independent sequences} $\phi_1,...,\phi_N:\N\rightarrow\Z$. The sequences $\phi_1,...,\phi_N$ are asymptotically (linearly) independent if for any $\vec a=(a_1,...,a_N)\in\Z^N\setminus \{\vec 0\}$,
$$\lim_{n\rightarrow\infty}|\sum_{j=1}^Na_j\phi_j(n)|=\infty.$$
\begin{thm}\label{0.IndependentPolysStepinResult}
Let $N\in\N$ and let $\lambda_1,...,\lambda_N\in [0,1]$.  For any asymptotically independent sequences $\phi_1,...,\phi_N:\N\rightarrow\Z$, the set
\begin{multline*}
    \{T\in\text{Aut}([0,1],\mathcal B,\mu)\,|\,\exists (n_k)_{k\in\N}\in\N^\N_\infty\,\forall j\in\{1,...,N\}\,\forall A,B\in\mathcal B,\\
    \lim_{k\rightarrow\infty}\mu(A\cap T^{- \phi_j(n_k)}B)=(1-\lambda_j)\mu(A\cap B)+\lambda_j\mu(A)\mu(B)\}
\end{multline*}
is a dense $G_\delta$ set.
\end{thm}
We will now formulate two results which are needed for the derivation  of \cref{0.IndependentPolysStepinResult}  (see \cref{0.MainInterpolationResult} and \cref{0.MainResultExistence} below).\\

The first of these results  is proved by utilizing a modified version of the "interpolation" techniques introduced in \cite{stepin1987spectral} and can be stated as follows:
\begin{thm}\label{0.MainInterpolationResult}
Let $N\in\N$, let $\lambda_1,...,\lambda_N\in[0,1]$, and let $\phi_1,...,\phi_N:\N\rightarrow\Z$. Suppose that $\phi_1,...,\phi_N$ satisfy the  following condition: 
\begin{adjustwidth}{0.5cm}{0.5cm}
\underline{Condition C}: There exists an $(n_k)_{k\in\N}\in\N^\N_\infty$ such that for any $\vec \xi=(\xi_1,...,\xi_N)\in\{0,1\}^N$, there exists an aperiodic  $T_{\vec \xi}\in\text{Aut}([0,1],\mathcal B,\mu)$ with the property that for each $j\in\{1,...,N\}$ and any $A,B\in\mathcal B$,
\begin{equation}\label{0.Mixing/RigidCondition}
\lim_{k\rightarrow\infty}\mu(A\cap T_{\vec \xi}^{- \phi_j(n_k)}B)=(1-\xi_j)\mu(A\cap B)+\xi_j\mu(A)\mu(B).
\end{equation}
\end{adjustwidth}
Then the set 
\begin{multline*}
\{T\in\text{Aut}([0,1],\mathcal B,\mu)\,|\exists (k_\ell)_{\ell\in\N}\in\N^\N_\infty\,\forall j\in\{1,...,N\}\,\forall A,B\in\mathcal B,\\
\lim_{\ell\rightarrow\infty}\mu(A\cap T^{- \phi_j(n_{k_\ell})}B)=(1-\lambda_j)\mu(A\cap B)+\lambda_j\mu(A)\mu(B)\}
\end{multline*}
is a dense $G_\delta$ set.
\end{thm}
To help the  reader appreciate the content of  \cref{0.MainInterpolationResult}, let us consider the case $N=1$. Fix an increasing sequence $(m_k)_{k\in\N}$ in $\N$ and set $\phi_1(k)=m_k$ for each  $k\in\N$. We claim that there exists an increasing sequence  $(n_k)_{k\in\N}$ in $\N$ for which $\phi_1$ satisfies Condition C. In other words there are transformations $T_0$ and $T_1$ such that for any $A,B\in\mathcal B$,
\begin{equation}\label{0.RigidTransformation}
\lim_{k\rightarrow\infty}\mu(A\cap T_0^{-\phi_1(n_k)}B)=\mu(A\cap B)
\end{equation}
and
\begin{equation}\label{0.MixingTransformation}
\lim_{k\rightarrow\infty}\mu(A\cap T_1^{-\phi_1(n_k)}B)=\mu(A)\mu(B).
\end{equation}
Note that the set $\bigcap_{q\in\N}\bigcup_{k\in\N}\{\alpha\in\R\,|\,|e^{2\pi i\phi_1(k)\alpha}-1|<\frac{1}{q}\}$ is a dense $G_\delta$ subset of $\R$. Thus, we can pick an irrational $\alpha$ and an increasing sequence $(n_k)_{k\in\N}$ such that  $\lim_{k\rightarrow\infty}(\phi_1(n_k)\alpha\,\text{mod}\,1)=0$. Letting  $T_0$ be the (aperiodic) transformation defined by $T_0(x)=(x+\alpha)\,\text{mod}\,1$ we have that $T_0$  satisfies \eqref{0.RigidTransformation}. 
Our claim now follows by noting that any strongly mixing transformation $T\in\text{Aut}([0,1],\mathcal B,\mu)$\footnote{
Let $(X,\mathcal F,\nu)$ be a probability space. A measure preserving transformation $T:X\rightarrow X$ is called strongly mixing if for any $A,B\in\mathcal F$,
$$\lim_{n\rightarrow\infty}\nu(A\cap T^{-n}B)=\nu(A)\nu(B).$$
} is aperiodic and satisfies \eqref{0.MixingTransformation}.\\  

The above discussion leads to the following corollary to  \cref{0.MainInterpolationResult}. (\cref{0.RefinementOfStepin} below is a  refinement of the result due  to St{\"e}pin mentioned above.)
\begin{cor}\label{0.RefinementOfStepin}
Let $(m_k)_{k\in\N}$ be an increasing sequence in $\N$ and let $\lambda\in[0,1]$. Then 
\begin{multline*}
\{T\in\text{Aut}([0,1],\mathcal B,\mu)\,|\,
\exists(k_\ell)_{\ell\in\N}\in\N^\N_\infty\,\forall A,B\in\mathcal B,\\
\lim_{\ell\rightarrow\infty}\mu(A\cap T^{-m_{k_\ell}}B)=(1-\lambda)\mu(A\cap B)+\lambda\mu(A)\mu(B)\}
\end{multline*}
is a dense $G_\delta$ set. 
\end{cor}
\begin{rem}
1. The special case of     \cref{0.RefinementOfStepin} corresponding to $\lambda=0$ gives an equivalent form  of Proposition 2.8 in \cite{BdJLRPublished} which states that given an increasing sequence $(m_k)_{k\in\N}$ in $\N$, the set 
\begin{equation*}
\{T\in\text{Aut}([0,1],\mathcal B,\mu)\,|\,
\exists(k_\ell)_{\ell\in\N}\in\N^\N_\infty\,\forall A,B\in\mathcal B,\,\lim_{\ell\rightarrow\infty}\mu(A\cap T^{-m_{k_\ell}}B)=\mu(A\cap B)\}
\end{equation*}
is residual. \\
2. The special case of     \cref{0.RefinementOfStepin} corresponding to $\lambda=1$ gives an equivalent form of the "folklore theorem" in  \cite[Proposition 2.14]{BdJLRArxivrigidity}, which states that given an increasing sequence $(m_k)_{k\in\N}$ in $\N$, the set 
\begin{equation*}
\{T\in\text{Aut}([0,1],\mathcal B,\mu)\,|\,
\exists(k_\ell)_{\ell\in\N}\in\N^\N_\infty\,\forall A,B\in\mathcal B,\,\lim_{\ell\rightarrow\infty}\mu(A\cap T^{-m_{k_\ell}}B)=\mu(A)\mu(B)\}
\end{equation*}
is residual. 
\end{rem}
As we will see below,  Condition C in \cref{0.MainInterpolationResult} is satisfied by any asymptotically independent sequences  $\phi_1,...,\phi_N:\N\rightarrow\Z$.  We remark in passing that for each  $N\geq 2$, there exist $\Q$-linearly dependent polynomials $v_1,...,v_N\in\Z[x]$ for which the (non-asymptotically independent) sequences $(\phi_j(k))_{k\in\N}=(v_j(k))_{k\in\N}$, $j\in\{1,...,N\}$, satisfy Condition C.  For instance, one can utilize the results in \cite{BKLUltrafilterPoly} to show that $\phi_1(n)=2n$ and $\phi_2(n)=3n$, $n\in\N$, satisfy Condition C. Moreover, one can deduce from \cite{BKLUltrafilterPoly}  that for any $N\geq 2$, the sequences 
$$\phi_j(n)=\left(\prod_{\{m\in\{1,...,2^N-2\}\,|\,j\in A_m\}}p_m\right)n,\,j\in\{1,...,N\},$$
where $A_1,...,A_{2^N-2}$ is an enumeration of the non-empty proper subsets of $\{1,...,N\}$ and $p_1,...,p_{2^N-2}$ are distinct prime numbers, satisfy Condition C (see also Section \ref{Section7OtherFamilies} of this paper).\\

The second result needed for the proof of \cref{0.IndependentPolysStepinResult} guarantees the existence of measure preserving transformations for which the sequences $\phi_1,...,\phi_N:\N\rightarrow\Z$ in \cref{0.IndependentPolysStepinResult} satisfy Condition C.
\begin{thm}\label{0.MainResultExistence}
Let $N\in\N$ and let $\phi_1,...,\phi_N:\N\rightarrow\Z$ be asymptotically independent sequences. Then there exists an increasing sequence $(n_k)_{k\in\N}$ in $\N$ such that for any $\vec \xi=(\xi_1,...,\xi_N)\in\{0,1\}^N$, there exists a weakly mixing\footnote{
Let $(X,\mathcal F,\nu)$ be a probability space. A measure preserving transformation $T:X\rightarrow X$ is called weakly mixing if for any $A,B\in\mathcal F$,
$$\lim_{N\rightarrow\infty}\frac{1}{N}\sum_{n=1}^N|\nu(A\cap T^{-n}B)-\nu(A)\nu(B)|=0.$$
Note that every weakly mixing transformation $S$ defined on $([0,1],\mathcal B,\mu)$ is aperiodic.
} 
$T_{\vec \xi}\in\text{Aut}([0,1],\mathcal B,\mu)$ with the property that for each $j\in\{1,...,N\}$ and any $A,B\in\mathcal B$,
\begin{equation}\label{0.KeyLimit}
\lim_{k\rightarrow\infty}\mu(A\cap T_{\vec \xi}^{- \phi_j(n_k)}B)=(1-\xi_j)\mu(A\cap B)+\xi_j\mu(A)\mu(B).
\end{equation}
\end{thm}
\begin{rem}
When $(\phi_j(k))_{k\in\N}=(v_j(k))_{k\in\N}$, $j\in\{1,...,N\}$, for some  $\Q$-linearly independent polynomials $v_1,...,v_N\in\Z[x]$  satisfying $v_j(0)=0$,  \cref{0.MainResultExistence} follows from Theorem 3.11 in \cite{BKLUltrafilterPoly}. We give an alternative proof of this restricted version of \cref{0.MainResultExistence} in Section \ref{Section3NoConstantTerm}.
\end{rem}

Consider now the polynomials $v_1,...,v_N\in\Z[x]$. We conclude this introduction by formulating a simple corollary of \cref{0.MainResultExistence} which links the linear independence of the polynomials $v_1(x)-v_1(0),...,v_N(x)-v_N(0)\in\Z[x]$ to the possible values of the limits of the form $$\lim_{k\rightarrow\infty}\mu(A\cap T^{- v_j(n_k)}B).$$
(Observe that the linear independence of the polynomials  $v_1(x)-v_1(0),...,v_N(x)-v_N(0)$ is equivalent to the asymptotic independence of the sequences $(v_1(k))_{k\in\N}$,...,$(v_N(k))_{k\in\N}$.)
\begin{cor}[Cf. Corollary F in \cite{BKLUltrafilterPoly}]\label{0.SimplePolyIndependence}
Let $N\in\N$ and let $t\in\{0,...,N\}$. For any non-constant polynomials $v_1,...,v_N\in\Z[x]$ such that $v_1(x)-v_1(0),...,v_N(x)-v_N(0)$ are $\Q$-linearly independent, there exists an increasing  sequence $(n_k)_{k\in\N}$ in $\N$ and a $T\in\text{Aut}([0,1],\mathcal B,\mu)$ with the property that for any $A,B\in\mathcal B$ and any $j\in\{1,...,N\}$,
 $$\lim_{k\rightarrow\infty}\mu(A\cap T^{- v_j(n_k)}B)=\begin{cases}
 \mu(A\cap B)\text{ if }j\leq t,\\
 \mu(A)\mu(B)\text{ if }j\in\{1,...,N\}\setminus\{0,...,t\}.
 \end{cases}$$
 \end{cor}
 The structure of this paper is as follows. In Section \ref{Section2GaussianBackground} we introduce the necessary background on Gaussian systems. In Section \ref{Section3NoConstantTerm} we prove a version of \cref{0.MainResultExistence} dealing with polynomials  having zero constant term. In Section \ref{Section4ArbitraryZPolynomials} we prove \cref{0.MainResultExistence}. The proof of the special case of \cref{0.MainResultExistence} given in Sections \ref{Section3NoConstantTerm} is  quite a bit simpler than and somewhat different from  the proof  of \cref{0.MainResultExistence} and is of interest on its own. In Section \ref{Section6InterpolationResults}   we prove \cref{0.MainInterpolationResult} and obtain \cref{0.IndependentPolysStepinResult} as a corollary. In Section \ref{Section7OtherFamilies} we use a slight modification of the methods introduced in Section \ref{Section3NoConstantTerm} to provide examples of non-asymptotically independent sequences for which Condition C holds. 
\section{Background on Gaussian Systems}\label{Section2GaussianBackground}
In this Section we review the necessary background material on Gaussian systems.
\subsection{Basic definitions}
Let $\mathcal A=\text{Borel}(\R^\Z)$ and consider the measurable space $(\R^\Z,\mathcal A)$. For each $n\in\Z$, we will let 
\begin{equation}\label{1.ProjectionDefn}
X_n:\R^\Z\rightarrow \R
\end{equation}
denote the projection onto the $n$th coordinate (i.e. for each $\omega\in\R^\Z$, $X_n(\omega)=\omega(n)$).\\
A non-negative Borel measure $\rho$ on $\mathbb T=[0,1)$ is called \textbf{symmetric} if for any $n\in\Z$,
$$\int_\mathbb Te^{2\pi inx}\text{d}\rho(x)=\int_\mathbb Te^{-2\pi inx}\text{d}\rho(x).$$
It is well known that for any symmetric non-negative finite Borel measure $\rho$ on $\mathbb T$, there exists a unique probability measure $\gamma=\gamma_{\rho}:\mathcal A\rightarrow [0,1]$ such that (a) for any 
$f\in H_1=\overline{\text{span}_{\R}\{X_n\,|\,n\in\Z\}^{L^2(\gamma)}}$, 
$f$ has a Gaussian distribution with mean zero\footnote{We will treat the constant function $f=0$ as a normal random variable with variance zero.} and (b)
for any $m,n\in\Z$,
\begin{equation}\label{1.CorrelationsCoincide}
\int_{\R^\Z} X_nX_m\text{d}\gamma=\int_{\mathbb T}e^{2\pi i(m-n)x}\text{d}\rho(x).
\end{equation}
We call the probability measure $\gamma$ the \textbf{Gaussian measure associated with $\rho$} and refer to $\rho$ as the \textbf{spectral measure associated with $\gamma$}. As we will see below, many of the properties of $\rho$ (and hence $H_1$) are intrinsically connected with those of $\gamma$.\\

Let $T:\R^\Z\rightarrow\R^\Z$ denote the shift map defined by
$$[T(\omega)](n)=\omega(n+1)$$
for each $\omega\in\R^\Z$ and each $n\in\Z$. The quadruple $(\R^\Z,\mathcal A,\gamma,T)$ is an invertible probability measure preserving system called the \textbf{Gaussian system associated with $\rho$}. (For the construction of a Gaussian system, see  \cite[Chapters 8]{cornfeld1982ergodic} or \cite[Appendix C]{kechris2010Global}, for example.)\\

Most of the results in the coming sections deal with non-trivial Gaussian systems. A Gaussian system $(\R^\Z,\mathcal A,\gamma,T)$ is \textbf{non-trivial} if its spectral measure is not the zero measure. (When $\rho$ is the zero measure, the associated Gaussian system is isomorphic to the probability measure preserving system with only one point.)
\subsection{Gaussian self-joinings of a Gaussian system}
In this subsection we review the necessary background material on Gaussian self-joinings of Gaussian systems, which were introduced in \cite{lemanczyk2000gaussian}.\\

A \textbf{self-joining} of a Gaussian system $(\R^\Z,\mathcal A,,\gamma, T)$ is a ($T\times T$)-invariant Borel probability measure $\Gamma:\mathcal A\otimes\mathcal A\rightarrow [0,1]$ such that for any $A\in\mathcal A$, $\Gamma(\R^\Z\times A)=\Gamma(A\times\R^\Z)=\gamma(A)$. Denote the set of all self-joinings of $(\R^\Z,\mathcal A,\gamma, T)$ by $\mathcal J(\gamma)$. 
Identifying $\gamma$ with a Borel probability measure on $[0,1]$, one can view $\mathcal J(\gamma)$ as a topological subspace of the space of all Borel probability measures on $[0,1]\times[0,1]$ with the weak-* topology. With this topology, $\mathcal J(\gamma)$ is a compact metrizable space with the property that for any  sequence $(\Gamma_k)_{k\in\N}$ in $\mathcal J(\gamma)$, 
$$\lim_{k\rightarrow\infty}\Gamma_k=\Gamma$$
if  and only if for every $A,B\in\mathcal A$,
\begin{equation}\label{1.JoiningConvergence}
\lim_{k\rightarrow\infty}\Gamma_k(A\times B)=\Gamma(A\times B).
\end{equation}
\begin{rem}
Condition \eqref{1.JoiningConvergence} is equivalent to the following (seemingly stronger) condition: For any $f,g\in L^2(\gamma)$, $$\lim_{k\rightarrow\infty}\int_{\R^\Z\times\R^\Z} f(\omega')g(\omega'')\text{d}\Gamma_k(\omega',\omega'')=\int_{\R^\Z\times\R^\Z} f(\omega')g(\omega'')\text{d}\Gamma(\omega',\omega'').$$
\end{rem}

Consider now the projections $X'_n,X''_n:\R^\Z\times\R^\Z\rightarrow \R$, $n\in\Z$, defined by
$$X'_n(\omega',\omega'')=\omega'(n)\text{ and }X_n''(\omega',\omega'')=\omega''(n)$$
for each $(\omega',\omega'')\in \R^\Z\times\R^\Z$. For any $\Gamma\in\mathcal J(\gamma)$, we will let $H_1'$ and $H_1''$ denote the closed real subspaces of $L^2(\Gamma)$ spanned by $(X'_n)_{n\in\Z}$ and $(X''_n)_{n\in\Z}$, respectively. Note that both $H_1'$ and $H_1''$ depend only on the topology of $L^2(\gamma)$ and not on the specific choice of $\Gamma$.\\
Given $\Gamma\in\mathcal J(\gamma)$, we say that $\Gamma$ is a \textbf{Gaussian self-joining} (of $(\R^\Z,\mathcal A,\gamma,T)$) if  $\overline{H'_1+H''_1}$ is a Gaussian  subspace in $L^2(\Gamma)$, meaning that for any $f\in \overline{H'_1+H''_1}$, $f$ has a Gaussian distribution. Denote the set of all Gaussian self-joinings of $\gamma$ by $\mathcal J_G(\gamma)$. One can show that for any $\Gamma\in\mathcal J_G(\gamma)$, $\Gamma$ is completely determined by the values of the correlations
$$\int_{\R^\Z\times\R^\Z}X_n'X_m''\text{d}\Gamma,\,n,m\in\Z.$$

The following are important examples of Gaussian self-joinings of $(\R^\Z,\mathcal A,\gamma,T)$:
\begin{enumerate}
    \item [-] The product measure $\gamma\otimes \gamma$. This measure is characterized by the correlations 
    \begin{equation}\label{1.IndependentCorrelations}
    \int_{\R^\Z\times\R^\Z}X_n'X_m''\text{d}\Gamma=0,\,n,m\in\Z.
    \end{equation}
    \item [-] The measure $\Delta_a$, $a\in\Z$, defined by $\Delta_a(A\times B)=\gamma(A\cap T^{-a}B)$ for any $A,B\in\mathcal A$. This measure is characterized by the correlations
     \begin{equation}\label{1.DependentCorrelations}
     \int_{\R^\Z\times\R^\Z}X_n'X_m''\text{d}\Gamma=\int_{\R^\Z}X_nX_{m+a}\text{d}\gamma,\,n,m\in\Z.
     \end{equation}
\end{enumerate}
The next proposition was mentioned as a consequence of Theorem 1 in \cite[page   267]{lemanczyk2000gaussian}.
\begin{prop}
$\mathcal J_G(\gamma)$ is a closed (and hence compact) subspace of $\mathcal J(\gamma)$.
\end{prop}
\begin{proof}
 Let $(\Gamma_k)_{k\in\N}$ be a sequence in $\mathcal J_G(\gamma)$ such that $\lim_{k\rightarrow \infty}\Gamma_k=\Gamma$ for some $\Gamma\in\mathcal J(\gamma)$. Since the limit of Gaussian distributions is again a Gaussian distribution, it suffices to show that for any  $f_1\in H_1'$ and $f_2\in H_1''$, the probability measure $\Gamma\circ (f_1+f_2)^{-1}$ has a Gaussian distribution. To prove this, we will compute the characteristic function $\phi$  of $\Gamma\circ(f_1+f_2)^{-1}$. For each $t\in\R$, 
 \begin{multline}\label{1.CharacteristicFunction}
 \phi(t)=\int_{\R^\Z\times\R^\Z} e^{ it[f_1(\omega')+f_2(\omega'')]}\text{d}\Gamma(\omega',\omega'')= \int_{\R^\Z\times\R^\Z} e^{ itf_1(\omega')}e^{itf_2(\omega'')}\text{d}\Gamma(\omega',\omega'')\\
 =\lim_{k\rightarrow\infty}\int_{\R^\Z\times\R^\Z} e^{ itf_1(\omega')}e^{itf_2(\omega'')}\text{d}\Gamma_k(\omega',\omega'')= \lim_{k\rightarrow\infty}\int_{\R^\Z\times\R^\Z} e^{ it[f_1(\omega')+f_2(\omega'')]}\text{d}\Gamma_k(\omega',\omega'').
 \end{multline}
For each $k\in\N$, $\Gamma_k\circ(f_1+f_2)^{-1}$ has a Gaussian distribution. Thus, by \eqref{1.CharacteristicFunction}, $\Gamma\circ(f_1+f_2)^{-1}$ has also a Gaussian distribution.
\end{proof}

\subsection{Connections between the mixing properties of $(\R^\Z,\mathcal A,\gamma,T)$ and its spectral measure}\label{Subsection2.2GaussianMixingProperties}
Before stating the results in this subsection, we need some definitions.\\

Let $(X,\mathcal F,\nu, S)$ be an invertible probability measure preserving system. We say that \textbf{$S$ has the mixing property along the sequence $(n_k)_{k\in\N}$ in $\Z$} if for any $A,B\in\mathcal F$,
$$\lim_{k\rightarrow\infty}\nu(A\cap S^{-n_k}B)=\nu(A)\nu(B).$$
We say that a system $(X,\mathcal F,\nu,S)$ is \textbf{rigid along the sequence $(n_k)_{k\in\N}$ in $\Z$} (or equivalently, $(n_k)_{k\in\N}$ is a rigidity sequence for $(X,\mathcal F,\nu,S)$) if for any $A,B\in\mathcal F$,
$$\lim_{k\rightarrow\infty}\nu(A\cap S^{-n_k}B)=\nu(A\cap B).$$ 

Now let $\rho$ be a positive finite Borel measure on $\mathbb T$ and let  $(n_k)_{k\in\N}$ be a sequence in $\Z$. We say that $\rho$ has the \textbf{mixing property along the sequence $(n_k)_{k\in\N}$ in $\Z$} if for every $m\in\Z$,
$$\lim_{k\rightarrow\infty}\int_\mathbb T e^{2\pi i (n_k+m)x}\text{d}\rho(x)=0.$$
We  say that $\rho$ is \textbf{rigid along the sequence $(n_k)_{k\in\N}$ in $\Z$} if for every $m\in\Z$,
$$\lim_{k\rightarrow\infty}\int_\mathbb T e^{2\pi i (n_k+m)x}\text{d}\rho(x)=\int_\mathbb T e^{2\pi imx}\text{d}\rho(x).$$

The following result exhibits the close connection between the "dynamical" properties of  a spectral measure $\rho$ defined on $\mathbb T$ and the Gaussian system associated with $\rho$.
\begin{thm}\label{1.NotionsOfMixing}
Let $\rho$ be a symmetric positive finite Borel measure on $\mathbb T$ and let
$(\R^\Z,\mathcal A,\gamma,T)$ be the Gaussian system associated with it. Given a  sequence $(n_k)_{k\in\N}$ in $\Z$, the following statements hold:
\begin{enumerate}[(i)]
    \item $T$ has the mixing property along $(n_k)_{k\in\N}$ if and only if 
    $\rho$ has the mixing property along $(n_k)_{k\in\N}$. 
    \item $T$ is rigid along $(n_k)_{k\in\N}$ if and only if  $\rho$ is rigid along $(n_k)_{k\in\N}$. 
    \item Let $a\in \Z$. The following are equivalent:
    \begin{enumerate}[(1)]
     \item For every $A,B\in\mathcal A$,
    \begin{equation}\label{1.a-rigidityForGaussians}
    \lim_{k\rightarrow\infty}\gamma(A\cap T^{-n_k}B)=\gamma(A\cap T^{-a}B).
    \end{equation}
     \item For every $m\in\Z$,
    \begin{equation}\label{1.a-rigidityForSpectralMeasure}
    \lim_{k\rightarrow\infty}\int_\mathbb Te^{2\pi i(n_k+m)x}\text{d}\rho(x)=\int_\mathbb T e^{2\pi i(a+m)x}\text{d}\rho(x).
    \end{equation}
    \end{enumerate}
\end{enumerate}
\end{thm}
\begin{proof}
The proofs of (i), (ii), and (iii) are similar. We will only prove (i).\\

 Suppose first that $T$ has the mixing property along $(n_k)_{k\in\N}$. Then, for any $m\in\Z$,
$$\lim_{k\rightarrow\infty}\int_\mathbb T e^{2\pi i(n_k+m)x}\text{d}\rho=\lim_{k\rightarrow\infty}\int_{\R^\Z}X_0T^{n_k}X_m\text{d}\gamma=\int_{\R^\Z}X_0\text{d}\gamma\int_{\R^\Z}X_m\text{d}\gamma=0.$$
Thus, $\rho$ has the mixing property along $(n_k)_{k\in\N}$.\\
Suppose now that $\rho$ has the mixing property along $(n_k)_{k\in\N}$. Let $(k_j)_{j\in\N}$ be an increasing sequence in $\N$ such that $\lim_{j\rightarrow\infty}\Delta_{n_{k_j}}=\Gamma$ for some $\Gamma\in\mathcal J_G(\gamma)$. For any $n,m\in\Z$, we have 
\begin{multline*}
\int_{\R^\Z\times\R^\Z}X_n'X_m''\text{d}\Gamma=\lim_{j\rightarrow\infty}\int_{\R^\Z\times\R^\Z}X'_nX''_m\text{d}\Delta_{n_{k_j}}=\lim_{j\rightarrow\infty}\int_{\R^\Z}X_nT^{n_{k_j}}X_m\text{d}\gamma\\
=\lim_{j\rightarrow\infty}\int_{\R^\Z}X_nX_{(n_{k_j}+m)}\text{d}\gamma=\lim_{j\rightarrow\infty}\int_\mathbb T e^{2\pi i(n_{k_{j}}+(m-n))x}\text{d}\rho=0.
\end{multline*}
Thus, by \eqref{1.IndependentCorrelations}, $\Gamma=\gamma\otimes\gamma$. It now follows from the compactness of $\mathcal J_G(\gamma)$, that $\lim_{k\rightarrow\infty}\Delta_{n_k}=\gamma\otimes\gamma$. In other words, for any $A,B\in\mathcal A$,
$$\lim_{k\rightarrow\infty}\gamma(A\cap T^{-n_k}B)=\lim_{k\rightarrow\infty}\Delta_{n_k}(A\times B)=\gamma\otimes \gamma(A\times B)=\gamma(A)\gamma(B).$$
We are done.
\end{proof}
We now record for future use the following classical result (see \cite[page 191]{cornfeld1982ergodic} and Theorem 1 in \cite[page 368]{cornfeld1982ergodic}, for example).
\begin{prop}\label{1.TheWienerConsequence}
Let $(\R^\Z,\mathcal A,\gamma, T)$ be a Gaussian system and let $\rho$ be the spectral measure associated with it. The following are equivalent: (i) $\rho$ is continuous. (ii) $T$ is weakly mixing. (iii) $T$ is ergodic.
\end{prop}
We conclude this section with an easy consequence of \cref{1.NotionsOfMixing} which illustrates the connection between non-trivial Gaussian systems  and $\text{Aut}([0,1],\mathcal B,\mu)$.
\begin{prop}\label{1.EquivalenceOftypesOfSystems}
Let $(n_k)_{k\in\N}$ be a sequence in $\Z$, let $\xi\in\{0,1\}$, and let $a\in\Z$. The following are equivalent:
\begin{enumerate}[(i)]
\item There exists a non-trivial Gaussian system $(\R^\Z,\mathcal A,\gamma,T)$ such that for any $A,B\in\mathcal A$,
\begin{equation}\label{1.LimitAlongSystem}
    \lim_{k\rightarrow\infty}\gamma(A\cap T^{-n_k}B)=(1-\xi)\gamma(A\cap T^{-a} B)+\xi\gamma(A)\gamma(B).
\end{equation}
\item There exists an $S\in\text{Aut}([0,1],\mathcal B,\mu)$ such that for any $A,B\in\mathcal B$,
\begin{equation}\label{1.LimitAlongTransformation}
    \lim_{k\rightarrow\infty}\mu(A\cap S^{-n_k}B)=(1-\xi)\mu(A\cap S^{-a} B)+\xi\mu(A)\mu(B).
\end{equation}
\end{enumerate}
\end{prop}
\begin{proof}
 (i)$\implies$(ii): Note that any non-trivial Gaussian system is measure theoretically isomorphic to $([0,1],\mathcal B,\mu,S)$ for some $S\in \text{Aut}([0,1],\beta,\mu)$ (see \cite[Theorem 2.1]{waltersIntroduction}, for example).\\
 (ii)$\implies$(i): Let $f\in L^2(\mu)$ be a non-zero real-valued function such that $\int_{[0,1]}f\text{d}\mu=0$ and let $\rho$ be the positive finite Borel measure satisfying
 $$\int_{[0,1]} fS^kf\text{d}\mu=\int_\mathbb T e^{2\pi i kx}\text{d}\rho(x)$$
 for each $k\in\Z$. Since $\int_\mathbb T e^{2\pi i kx}\text{d}\rho(x)$ is a real number for each $k\in\Z$, we have that $\rho$ is symmetric.\\
 By \eqref{1.LimitAlongTransformation}, for any $g\in L^2(\mu)$,
 $$\lim_{k\rightarrow\infty}\int_{[0,1]}gS^{n_k}f\text{d}\mu=(1-\xi)\int_{[0,1]}gS^af\text{d}\mu+\xi\int_{[0,1]}g\text{d}\mu\int_{[0,1]}f\text{d}\mu.$$
 Thus, for any $m\in\Z$,
 \begin{multline*}
 \lim_{k\rightarrow\infty}\int_\mathbb T e^{2\pi i (n_k+m)x}\text{d}\rho(x)=\lim_{k\rightarrow\infty}\int_{[0,1]}fS^{n_k+m}f\text{d}\mu=\lim_{k\rightarrow\infty}\int_{[0,1]}S^{-m}fS^{n_k}f\text{d}\mu\\
 =(1-\xi)\int_{[0,1]}S^{-m}fS^af\text{d}\mu+\xi\int_{[0,1]}S^{-m}f\text{d}\mu\int_{[0,1]}f\text{d}\mu=(1-\xi)\int_{[0,1]}fS^{a+m}f\text{d}\mu\\
 =(1-\xi)\int_\mathbb T e^{2\pi i(a+m)x}\text{d}\rho(x).
 \end{multline*}
 Taking $(\R^\Z,\mathcal A,\gamma, T)$ to be the non-trivial Gaussian system associated with $\rho$ in (i), we see that \eqref{1.LimitAlongSystem} holds.
\end{proof}
\section{A version of \cref{0.MainResultExistence} for polynomials having zero constant term}\label{Section3NoConstantTerm}
In this section we prove a special case of \cref{0.MainResultExistence}  which deals with polynomials $v_1,...,v_N$ in $\Z[x]$ satisfying $v_j(0)=0$ for each $j\in\{1,...,N\}$. It will be stated in the language of Gaussian systems (see \cref{3.IndependentPoly} below). Unlike the proof of \cref{0.MainResultExistence} in its full generality, the proof of this special case  utilizes a simple and explicit construction for the spectral measures associated with each of the Gaussian systems guaranteed to exist in \cref{3.IndependentPoly}. As demonstrated in \cite[Proposition 7.1]{BerZel_JCTA_iteratedDifferences2021} and in Section \ref{Section7OtherFamilies} of  this paper, this method can be used to provide examples of measure preserving systems with various kinds of asymptotic behavior.
We remark that while \cref{0.MainResultExistence} deals with automorphisms of $[0,1]$ the formulation of \cref{3.IndependentPoly} deals  with non-trivial Gaussian systems $(\R^\Z,\mathcal A,\gamma,T)$. This distinction is immaterial due to a slight modification of  \cref{1.EquivalenceOftypesOfSystems}.

\begin{thm}[Cf. \cref{0.MainResultExistence}]\label{3.IndependentPoly}
Let $N\in\N$, let $(m_k)_{k\in\N}$ be an increasing sequence in $\N$ with $k!|m_k$ for each $k\in\N$,  and let the non-constant polynomials $v_1,...,v_N\in\Z[x]$ be $\Q$-linearly independent and such that for each $j\in\{1,...,N\}$, $v_j(0)=0$. Then there exists a subsequence  $(n_k)_{k\in\N}$ of $(m_k)_{k\in\N}$ such that for any $\vec \xi=(\xi_1,...,\xi_N)\in\{0,1\}^N$, there exists a non-trivial weakly mixing Gaussian system  $(\R^\Z,\mathcal A,\gamma_{\vec \xi},T_{\vec \xi})$ with the property that for each $j\in\{1,...,N\}$ and any $A,B\in\mathcal A$,
\begin{equation}\label{3.MixingOrRigidExpression}
\lim_{k\rightarrow\infty}\gamma_{\vec\xi}(A\cap T_{\vec \xi}^{-v_j(n_k)}B)=(1-\xi_j)\gamma_{\vec \xi}(A\cap B)+\xi_j\gamma_{\vec \xi}(A)\gamma_{\vec \xi}(B).
\end{equation}
\end{thm}
\begin{proof}
By \cref{1.NotionsOfMixing} and \cref{1.TheWienerConsequence}, it suffices to show that there exist a subsequence $(n_k)_{k\in\N}$ of  $(m_k)_{k\in\N}$ and continuous Borel probability measures $\sigma_{\vec \xi}$ on $\mathbb T=[0,1)$, $\vec \xi\in\{0,1\}^N$,  such that for each $\vec\xi=(\xi_1,...,\xi_N)\in\{0,1\}^N$, the  sequence
$$a^{(\vec \xi)}_k=\int_\mathbb T e^{2\pi ikx}\text{d}\sigma_{\vec \xi}(x),\,k\in\Z,$$
is a real-valued sequence with $a^{(\vec \xi)}_0=1$ (which implies that $\sigma_{\vec \xi}$ is symmetric and non-zero),  and for each $j\in\{1,...,N\}$ and any $m\in\Z$,
\begin{equation}\label{3.TargetProperties}
    \lim_{k\rightarrow\infty}\int_\mathbb T e^{2\pi i (v_j(n_k)+m)x}\text{d}\sigma_{\vec \xi}(x)=(1-\xi_j)\int_\mathbb T e^{2\pi i mx}\text{d}\sigma_{\vec \xi}(x).
\end{equation}
We now proceed to construct the probability measures $\sigma_{\vec \xi}$, $\vec \xi\in\{0,1\}^N$, with the desired properties. Let
$$d=\max _{1\leq j\leq N}\deg v_j$$
and for $j\in\{1,...,N\}$, let $a_{j,1},...,a_{j,d}\in\Z$ be such that
\begin{equation}
    v_j(x)=\sum_{s=1}^da_{j,s}x^s.
\end{equation}
We define the  $N\times d$ matrix $D$ by  
\begin{equation}\label{3.Adefn}
(D)_{j,s}=a_{j,s}
\end{equation}
for $j\in\{1,...,N\}$ and $s\in\{1,...,d\}$. For each $j\in\{1,...,N\}$ and each $\vec \xi=(\xi_1,...,\xi_N)\in\{0,1\}^N$, let $b^{(\vec \xi)}_j=1-\frac{\xi_j}{2}$  and set 
\begin{equation}\label{3.SolutionVector}
    \vec b_{\vec \xi}=(b^{(\vec \xi)}_1,...,b^{(\vec \xi)}_N).
\end{equation}
Since $v_1,...,v_N$ are linearly independent, the rank of $D$ is $N$. Hence, for each $\vec \xi\in\{0,1\}^N$, there exists a non-zero $\vec x_{\vec \xi}=(x^{(\vec \xi)}_1,...,x^{(\vec \xi)}_d)\in \Q^d$ satisfying
\begin{equation}\label{3.DefnOfX}
D\vec x_{\vec \xi}=\vec b_{\vec \xi}.
\end{equation}
Let $n_0\in\N$ be such that $n_0>1$. Choose a subsequence $(n_k)_{k\in\N}$ of $(m_k)_{k\in\N}$ with the property that for any $\vec \xi=(\xi_1,...,\xi_N)\in\{0,1\}^N$, any $j\in\{1,...,d\}$, and any $k\in\N$,  (a) $dn_0|x_j^{(\vec \xi)}|<n_1$, (b) $x_j^{(\vec \xi)}n_k\in\Z$,  and (c) $(2dn_{k-1}^{d+1})|n_{k}$.\\ 
Let $\{0,1\}^\N$ be endowed with the product topology and let $\mathbb P$ be the $(\frac{1}{2},\frac{1}{2})$-probability measure on $\{0,1\}^\N$. For each $\vec \xi\in\{0,1\}^N$, we define $f_{\vec \xi}:\{0,1\}^\N\times \{0,1\}^\N\rightarrow\mathbb T$ by 
\begin{equation}\label{3.f_xiDefn}
f_{\vec  \xi}(\omega_1,\omega_2)=\sum_{t=1}^\infty\sum_{s=1}^d\frac{x^{(\vec \xi)}_s}{n_t^s}(\omega_1(t)-\omega_2(t))\mod 1.
\end{equation}
Since for any $\omega_1,\omega_2\in\{0,1\}^\N$ and any $k\in\N$, $|\omega_1(k)-\omega_2(k)|\leq1$, (a) implies that for any $t\in\N$ and any $s\in\{1,...,d\}$,
$$\left|\frac{x^{(\vec \xi)}_s}{n_t^s}(\omega_1(t)-\omega_2(t))\right|\leq \frac{|x^{(\vec \xi)}_s|}{n_t^s}\leq \frac{n_1}{dn_0n_t^s}.$$ 
By (c), 
\begin{equation}\label{3.Bound}
 \sum_{t=1}^\infty\sum_{s=1}^d\frac{n_1}{dn_0n_t^s}\leq \sum_{t=1}^\infty\frac{dn_1}{dn_0n_t}
=\frac{1}{n_0}\sum_{t=1}^\infty\frac{n_1}{n_t}\leq \frac{1}{n_0}\sum_{t=0}^\infty\frac{1}{n_1^t}=\frac{1}{n_0}\frac{n_1}{n_1-1}\leq 1.
\end{equation}
Thus, by Weierstrass M-test, the function  $g_{\vec \xi}:\{0,1\}^\N\times\{0,1\}^\N\rightarrow \R$ given by 
$$g_{\vec \xi}(\omega_1,\omega_2)= \sum_{t=1}^\infty\sum_{s=1}^d\frac{x^{(\vec \xi)}_s}{n_t^s}(\omega_1(t)-\omega_2(t))$$
is well-defined and continuous.\\
Let $\phi$ be the canonical map from $\R$ to $[0,1)=\R/\Z$ (so $\phi(x)=x\mod\,1$ and $\phi$ is continuous). Since  $f_{\vec \xi}=\phi\circ g_{\vec \xi}$, we have that $f_{\vec \xi}$ is continuous and hence measurable. For each $\vec \xi \in \{0,1\}^N$, we will let 
$$\sigma_{\vec \xi}=(\mathbb P\times \mathbb P)\circ f_{\vec \xi}^{-1}.$$

Fix now $\vec \xi=(\xi_1,...,\xi_N)\in\{0,1\}^N$. Clearly $\sigma_{\vec \xi}$ is a Borel probability measure on $\mathbb T$ (and so,  $a^{(\vec \xi)}_0=1$). All it remains to show is that (i) $\sigma_{\vec \xi}$ is continuous, (ii) that $(a^{(\vec \xi)}_k)_{k\in\Z}$ is real-valued,  and (iii) that $\sigma_{\vec \xi}$ satisfies \eqref{3.TargetProperties}. For this let $f:\{0,1\}^\N\rightarrow\R$ be defined by 
$$f(\omega)=\sum_{t=1}^\infty\sum_{s=1}^d\frac{x^{(\vec \xi)}_s}{n_t^s}\frac{\omega(t)}{2}.$$
(Note that by an inequality similar to \eqref{3.Bound}, one can show that $f$ is well-defined and continuous).\\

\underline{(i)}: We will now show that $\sigma_{\vec \xi}$ is continuous, but first we need some estimates.\\ 

Combining (b) and (c) we obtain that for each $\ell\in\{1,...,d\}$, each $\omega\in\{0,1\}^\N$, and each $k>1$,
\begin{multline}\label{3.FirstTermsAreIntegers} 
n_k^{\ell}f(\omega)\,\text{mod}\,1      \equiv   n_k^\ell(\sum_{t=1}^\infty\sum_{s=1}^d\frac{x^{(\vec \xi)}_s}{n_t^s}\frac{\omega(t)}{2})\equiv   \sum_{t=1}^\infty\sum_{s=1}^d\frac{n_k^\ell x^{(\vec \xi)}_s}{n_t^s}\frac{\omega(t)}{2}\\
\equiv      \underbrace{\sum_{t=1}^{k-1}\sum_{s=1}^d\frac{n_k^\ell x^{(\vec \xi)}_s}{n_t^{s}}\frac{\omega(t)}{2})}_{\text{This is an integer}}+\underbrace{\sum_{s=1}^{\ell-1}\frac{n_k^\ell x^{(\vec \xi)}_s}{n_k^{s}}\frac{\omega(k)}{2}}_{\text{This 
is an integer}}+\sum_{s=\ell}^d\frac{x^{(\vec 
\xi)}_s}{n_k^{s-\ell}}\frac{\omega(k)}{2}+\sum_{t=k+1}^\infty\sum_{s=1}^d\frac{n_k^\ell x^{(\vec 
\xi)}_s}{n_t^{s}}\frac{\omega(t)}{2}\\
\equiv   x^{(\vec \xi)}_\ell\frac{\omega(k)}{2}+\sum_{s=\ell+1}^d\frac{x^{(\vec \xi)}_s}{n_k^{s-\ell}}\frac{\omega(k)}{2}+\sum_{t=k+1}^\infty\sum_{s=1}^d\frac{n_k^\ell x^{(\vec \xi)}_s}{n_t^{s}}\frac{\omega(t)}{2}\mod 1.
\end{multline}
By  (a) and  (c) we have
\begin{multline}\label{3.BoundOnTheTail}
|\sum_{s=\ell+1}^d\frac{x^{(\vec \xi)}_s}{n_k^{s-\ell}}\frac{\omega(k)}{2}+\sum_{t=k+1}^\infty\sum_{s=1}^d\frac{n_k^\ell x^{(\vec \xi)}_s}{n_t^{s}}\frac{\omega(t)}{2}|\leq \sum_{s=\ell+1}^d\frac{n_1}{n_k^{s-\ell}}+\sum_{t=k+1}^\infty\frac{n_k^\ell n_1}{n_t}\\
\leq \sum_{t=1}^d\frac{n_1}{n_k^t}+\sum_{t=1}^\infty\frac{n_1}{n_k^t}\leq 2n_1\sum_{t=1}^\infty\frac{1}{n_k^t}=\frac{2n_1}{n_k-1}.
\end{multline}
(Note that when $\ell=d$, $|\sum_{t=k+1}^\infty\sum_{s=1}^d\frac{n_k^\ell x^{(\vec \xi)}_s}{n_t^{s}}\frac{\omega(t)}{2}|<\frac{2n_1}{n_k-1}$ also holds.)\\

Denote the distance to the closest integer by $\|\cdot\|$ (so for any $r\in\R$, $\|r\|=\inf_{n\in\Z}|r-n|$ and, in particular, $\|r\|\leq |r|$). Consider a polynomial with integer coefficients $v(n)=\sum_{\ell=1}^da_\ell n^\ell$. By \eqref{3.FirstTermsAreIntegers} and \eqref{3.BoundOnTheTail}, for any $k>1$ and any $\omega\in\{0,1\}^\N$, 
\begin{multline*}
\|v(n_k)f(\omega)-\sum_{\ell=1}^da_{\ell}x^{(\vec \xi)}_\ell\frac{\omega(k)}{2}\|
=\|\sum_{\ell=1}^d a_\ell[n_k^{\ell}f(\omega)-x^{(\vec \xi)}_\ell\frac{\omega(k)}{2}]\| 
\leq \sum_{\ell=1}^d|a_\ell|\left\|n_k^\ell f(\omega)-x^{(\vec \xi)}_\ell\frac{\omega(k)}{2}\right\|\\
=\sum_{\ell=1}^d|a_{\ell}|\left\|\sum_{s=\ell+1}^d\frac{x^{(\vec \xi)}_s}{n_k^{s-\ell}}\frac{\omega(k)}{2}+\sum_{t=k+1}^\infty\sum_{s=1}^d\frac{n_k^\ell x^{(\vec \xi)}_s}{n_t^{s}}\frac{\omega(t)}{2}\right\|\\
\leq  \sum_{\ell=1}^d|a_{\ell}|\left|\sum_{s=\ell+1}^d\frac{x^{(\vec \xi)}_s}{n_k^{s-\ell}}\frac{\omega(k)}{2}+\sum_{t=k+1}^\infty\sum_{s=1}^d\frac{n_k^\ell x^{(\vec \xi)}_s}{n_t^{s}}\frac{\omega(t)}{2}\right| \leq \sum_{\ell=1}^d|a_{\ell}|\left(\frac{2n_1}{n_k-1}\right).
\end{multline*}
Thus, for any $\epsilon>0$ there exists $k_\epsilon\in\N$ such  that for any $k>k_\epsilon$ and any  $\omega\in\{0,1\}^\N$,
\begin{equation}\label{3.ImportantForTheRemark}
\|v(n_k)f(\omega)-\sum_{\ell=1}^da_{\ell}x^{(\vec \xi)}_\ell\frac{\omega(k)}{2}\|<\epsilon.
\end{equation}

Pick now $\alpha\in\R$ and suppose that there exists an $\omega_\alpha\in\{0,1\}^\N$ such that $f(\omega_\alpha)\equiv\alpha\mod 1$.  By \eqref{3.ImportantForTheRemark}, there exists $k_{1/8}\in\N$ such that for any $k>k_{1/8}$ and any $\omega\in\{0,1\}^\N$ with $f(\omega)\equiv\alpha\mod 1$,
\begin{multline*}\label{3.EqualAfterSomePoint}
\|(1-\frac{\xi_1}{2})\frac{\omega(k)-\omega_\alpha(k)}{2}\|=\|b_1^{(\vec \xi)}\frac{\omega(k)-\omega_\alpha(k)}{2}\|=\|\sum_{\ell=1}^da_{1,\ell}x^{(\vec \xi)}_\ell\frac{\omega(k)-\omega_\alpha(k)}{2}\|\\
\leq \|\sum_{\ell=1}^da_{1,\ell}x^{(\vec \xi)}_\ell\frac{\omega(k)-\omega_\alpha(k)}{2}-v_1(n_k)(f(\omega)-f(\omega_\alpha))\|+\|v_1(n_k)(f(\omega)-f(\omega_\alpha))\| \\
\leq \|\sum_{\ell=1}^da_{1,\ell}x^{(\vec \xi)}_\ell\frac{\omega(k)}{2}-v_1(n_k)f(\omega)\|+\|\sum_{\ell=1}^da_{1,\ell}x^{(\vec \xi)}_\ell\frac{\omega_\alpha(k)}{2}-v_1(n_k)f(\omega_\alpha)\|+\|v_1(n_k)(f(\omega)-f(\omega_\alpha))\|\\
<\frac{1}{8}+\frac{1}{8}+ \|v_1(n_k)(f(\omega)-f(\omega_\alpha))\|=\frac{1}{4}+0=\frac{1}{4},
\end{multline*}
Note that if $\omega(k)\neq \omega_\alpha(k)$, then $|(1-\frac{\xi_1}{2})\frac{\omega(k)-\omega_\alpha(k)}{2}|\in\{\frac{1}{4},\frac{1}{2}\}$ . Since for any $k>k_{1/8}$, $$|(1-\frac{\xi_1}{2})\frac{\omega(k)-\omega_\alpha(k)}{2}|=\|(1-\frac{\xi_1}{2})\frac{\omega(k)-\omega_\alpha(k)}{2}\|<\frac{1}{4},$$
we have $|(1-\frac{\xi_1}{2})\frac{\omega(k)-\omega_\alpha(k)}{2}|\not\in\{\frac{1}{4},\frac{1}{2}\}$
and hence $\omega(k)=\omega_\alpha(k)$.
It follows that $f^{-1}(\{\alpha+n\,|\,n\in\Z\})$ is a subset of
$$\{\omega\in\{0,1\}^\N\,|\,\forall k>k_{1/8},\,\omega(k)=\omega_\alpha(k)\},$$
which has at most $2^{k_{1/8}}$ elements.\\

Let $g:\{0,1\}^\N\rightarrow\mathbb T$ be defined by 
$$g(\omega)=2f(\omega)\,\text{mod}\,1=\sum_{t=1}^\infty\sum_{s=1}^d\frac{x^{(\vec \xi)}_s}{n_t^s}\omega(t)\mod 1.$$
and set 
$$\rho=\mathbb P\circ g^{-1}.$$
Take $\alpha\in [0,1)$ and let $x=\frac{\alpha}{2}$. Regarding $\alpha$ as an element of $\mathbb T=\R/\Z$, we have 
$$g^{-1}(\{\alpha\})=f^{-1}(\{x+n\,|\,n\in\Z\})\cup f^{-1}(\{x+\frac{1}{2}+n\,|\,n\in\Z\}).$$
It follows that $g^{-1}(\{\alpha\})$ is finite and hence
$$\rho(\{\alpha\})=\mathbb P(g^{-1}(\{\alpha\}))=0.$$
Noting that $f_{\vec \xi}(\omega_1,\omega_2)=g(\omega_1)-g(\omega_2)$, we have
\begin{multline*}
\sigma_{\vec \xi}(\{\alpha\})=\int_\mathbb T \mathbbm 1_{\{\alpha\}}(x)\text{d}\sigma_{\vec \xi}(x)=\int_{\{0,1\}^\N}\int_{\{0,1\}^\N}\mathbbm 1_{\{\alpha\}}(f_{\vec \xi}(\omega_1,\omega_2))\text{d}\mathbb P(\omega_1)\text{d}\mathbb P(\omega_2)\\
=\int_{\{0,1\}^\N}\int_{\{0,1\}^\N}\mathbbm 1_{\{\alpha\}}(g(\omega_1)-g(\omega_2))\text{d}\mathbb P(\omega_1)\text{d}\mathbb P(\omega_2)=\int_\mathbb T\int_\mathbb T \mathbbm 1_{\{\alpha\}}(x-y)\text{d}\rho(x)\text{d}\rho(y)=0.
\end{multline*}
So, $\sigma_{\vec \xi}$ is continuous.\\

\underline{(ii)}: For each $m\in\Z$,
\begin{equation}\label{3.FromSigmaTorho}
\int_\mathbb T e^{2\pi imx}\text{d}\sigma_{\vec \xi}(x)=\int_\mathbb T \int_\mathbb T e^{2\pi i m(x-y)}\text{d}\rho(x)\text{d}\rho(y)=|\int_\mathbb T e^{2\pi im x}\text{d}\rho(x)|^2.
\end{equation}
Thus the sequence $(a_k^{(\vec\xi)})_{k\in\Z}$ is real-valued.\\

\underline{(iii)}: Finally, we show that $\sigma_{\vec \xi}$ satisfies \eqref{3.TargetProperties}.
By \eqref{3.ImportantForTheRemark} and the definitions of $g$, $D$, $\vec x_{\vec \xi}$, and $\vec b_{\vec \xi}$, each $j\in\{1,...,N\}$ and each $m\in\Z$ satisfies
\begin{multline}\label{3.FirstMainEquality}
\lim_{k\rightarrow\infty}\int_\mathbb T e^{2\pi i (v_j(n_k)+m)x}\text{d}\rho(x)=\lim_{k\rightarrow\infty}\int_{\{0,1\}^\N} e^{2\pi i (v_j(n_k)+m)g(\omega)}\text{d}\mathbb P(\omega)\\=\lim_{k\rightarrow\infty}\int_{\{0,1\}^\N} e^{2\pi i (v_j(n_k)+m)2f(\omega)}\text{d}\mathbb P(\omega)
=\lim_{k\rightarrow\infty}\int_{\{0,1\}^\N} e^{2\pi i [2v_j(n_k)f(\omega)]}e^{2\pi i[2m f(\omega)]}\text{d}\mathbb P(\omega)\\
=\lim_{k\rightarrow\infty}\int_{\{0,1\}^\N} e^{2\pi i (2(1-\frac{\xi_j}{2})\frac{\omega(k)}{2})}e^{2\pi i[2mf(\omega)]}\text{d}\mathbb P(\omega)
=\lim_{k\rightarrow\infty}\int_{\{0,1\}^\N} e^{2\pi i(1-\frac{\xi_j}{2})\omega(k)}e^{2\pi i[2mf(\omega)]}\text{d}\mathbb P(\omega),
\end{multline}
whenever any (and hence each) of the limits in \eqref{3.FirstMainEquality} exists.\\
Since the shift map on $\{0,1\}^\N$  is $\mathbb P$-mixing, the last expression in \eqref{3.FirstMainEquality} can be rewritten as
\begin{equation}\label{3.KeyExpression}
\int_{\{0,1\}^\N} e^{2\pi i(1-\frac{\xi_j}{2})\omega(1)}\text{d}\mathbb P(\omega)\int_{\{0,1\}^\N} e^{2\pi i[2mf(\omega)]}\text{d}\mathbb P(\omega).
\end{equation}
Since $\omega(1)$ equals each of $1$ and $0$  with probability $\frac{1}{2}$, we get that \eqref{3.KeyExpression} equals
    $$\sum_{r=0}^1\frac{e^{2\pi i\frac{\xi_jr}{2}}}{2}\int_\mathbb T e^{2\pi imx}\text{d}\rho(x)=\begin{cases}
    0,\text{ if }\xi_j=1,\\
    \int_\mathbb T e^{2\pi imx}\text{d}\rho,\text{ if }\xi_j=0,
    \end{cases}.$$
So, by \eqref{3.FromSigmaTorho}, 
\begin{multline*}
    \lim_{k\rightarrow\infty}\int_\mathbb T e^{2\pi i (v_j(n_k)+m)x}\text{d}\sigma_{\vec \xi}(x)=\lim_{k\rightarrow\infty}|\int_\mathbb T e^{2\pi i (v_j(n_k)+m)x}\text{d}\rho(x)|^2\\
    =|(1-\xi_j)\int_\mathbb T e^{2\pi imx}\text{d}\rho(x)|^2=(1-\xi_j)|\int_\mathbb T e^{2\pi imx}\text{d}\rho(x)|^2=(1-\xi_j)\int_\mathbb T e^{2\pi imx}\text{d}\sigma_{\vec \xi}(x),
\end{multline*}
proving that \eqref{3.TargetProperties} holds. 
\end{proof}
\section{The proof of \cref{0.MainResultExistence}}\label{Section4ArbitraryZPolynomials}
In this section we prove  \cref{0.MainResultExistence} (=\cref{4.MainResultExistence} below) on its full generality. First, we need a technical lemma.\\
 
Given any  sequences $\phi_1,...,\phi_N:\N\rightarrow\Z$, we say that the sequences $\phi_1,...,\phi_N$ are \textbf{strongly asymptotically independent} if for any $\vec a=(a_1,...,a_N)\in \Z^N\setminus\{\vec 0\}$, the sequence 
$$a_1\phi_1(k)+\cdots+a_N\phi_N(k),\,k\in\N$$
is eventually a strictly monotone sequence (so, in particular, $\lim_{k\rightarrow\infty}|\sum_{s=1}^Na_s\phi_s(k)|=\infty$).
\begin{lem}\label{4.WeylsGeneralization}[Cf. Theorem 21 in \cite{weyl1916Mod1}]
Let $\phi_1,...,\phi_N:\N\rightarrow\Z$ be strongly asymptotically independent sequences. For any $t\in\N$, the set
\begin{multline*}
\mathfrak M_t(\phi_1,...,\phi_N)=\\
\{(\alpha_1,...,\alpha_t)\in \R^t\,|\,(\phi_1(k)\alpha_1,...,\phi_N(k)\alpha_1,...,\phi_1(k)\alpha_t,...,\phi_N(k)\alpha_t)_{k\in\N}\,\text{ is u.d. mod}\,1\}
\end{multline*}
has full Lebesgue measure on $\R^t$. Furthermore, for any $(\alpha_1,...,\alpha_t)\in \mathfrak M_t(\phi_1,...,\phi_N)$, the set 
\begin{equation}\label{4.KeyWeylsTypeResult} 
\mathfrak M(\phi_1,...,\phi_N,\alpha_1,...,\alpha_t)=\{\alpha\in\R\,|\,(\alpha_1,...,\alpha_t,\alpha)\in \mathfrak M_{t+1}(\phi_1,...,\phi_N)\}
\end{equation}
has full measure on $\R$.
\end{lem}
\begin{proof}
To prove the first claim  we will use induction on $t\in\N$. When $t=1$, the proof is the same as that of Theorem 4.1 in \cite{kuipers2012uniform}.  By Weyl's criterion for uniform distribution mod 1,  it suffices to show that for any $(a_1,...,a_N)\in\Z^N\setminus\{\vec 0\}$ the set,
$$\{\alpha\in [0,1)\,|\,\lim_{M\rightarrow\infty}\frac{1}{M}\sum_{r=1}^M \exp[2\pi i\sum_{j=1}^N a_j\phi_j(r)\alpha]=0\}$$
has Lebesgue measure 1.\\
For each $M\in\N$ and each $\alpha \in [0,1)$ define  
$$S(M)(\alpha)=\frac{1}{M}\sum_{r=1}^M \exp[2\pi i\sum_{j=1}^N a_j\phi_j(r) \alpha].$$
Observe that 
\begin{equation}\label{4.sqrNorm}
||S(M)||_{L^2(\mathbb T)}^2=
\frac{1}{M^2}\sum_{r,s=1}^M\int_{\mathbb T}\exp[2\pi i\sum_{j=1}^N a_j(\phi_j(r)-\phi_j(s)) x]\text{d}x.
\end{equation}
The right hand side on \eqref{4.sqrNorm} can be written as
\begin{equation}\label{4.sqrtNorm2}
\frac{1}{M}+\frac{1}{M^2}\sum_{s>r=1}^M2\text{Re}\left(\int_{\mathbb T}\exp[2\pi i\sum_{j=1}^N a_j(\phi_j(s)-\phi_j(r))x]\text{d}x\right).
\end{equation}
So, since $\phi_1,...,\phi_N$ are strongly asymptotically independent, it follows from  \eqref{4.sqrtNorm2} that for  $M\in\N$ large enough
$$||S(M)||_{L^2(\mathbb T)}^2<\frac{2}{M}.$$
It follows that 
$$\int_{\mathbb T} \sum_{M=1}^\infty |S(M^2)(x)|^2\text{d}x=\sum_{M=1}^\infty||S(M^2)||_{L^2(\mathbb T)}^2<\infty$$
and hence, for almost every $\alpha\in \mathbb T$, $\sum_{M=1}^\infty |S(M^2)(\alpha)|^2 <\infty$.\\
So, in particular,  for almost every $\alpha\in \mathbb T$,
\begin{equation}\label{4.LimitAlongSquares}
    \lim_{M\rightarrow \infty} S(M^2)(\alpha)=0.
\end{equation}

We will now show that \eqref{4.LimitAlongSquares} implies that for almost every $\alpha\in \mathbb T$, $$\lim_{M\rightarrow\infty} S(M)(\alpha)=0.$$
Indeed, let $\alpha\in\mathbb T$ be such that $\lim_{M\rightarrow \infty} S(M^2)(\alpha)=0$ and let  $M,M_0\in\N$ satisfy $M_0^2\leq M<(M_0+1)^2$. Since  
$$|S(M)(\alpha)-S(M_0^2)(\alpha)|
\leq \frac{1}{M_0^2}\sum_{n=1}^{M_0^2}\left(1-\frac{M_0^2}{M}\right)+\frac{1}{M}\sum_{n=M_0^2+1}^{M}1,$$
we have that
$$|S(M)(\alpha)-S(M_0^2)(\alpha)|\leq 1-\frac{M_0^2}{M}+\frac{2M_0+1}{M}\leq 1-\frac{M_0^2}{(M_0+1)^2}+\frac{2M_0+1}{M_0^2}.$$
Thus, 
$$\lim_{M\rightarrow\infty} S(M)(\alpha)=0,$$
proving that $\mathfrak M_1(\phi_1,...,\phi_N)$ has full Lebesgue measure in $\R$.\\

Now let $t\in\N$ and suppose that for any $t'\leq t$ and any strongly asymptotically independent $g_1,...,g_N:\N\rightarrow\Z$, $\mathfrak M_{t'}(g_1,...,g_N)$ has full measure in $\R^{t'}$. We want to show that $\mathfrak M_{t+1}(\phi_1,...,\phi_N)$ has full measure in $\R^{t+1}$.\\
For each $R\in \N$ and each $\vec r=(r_{1,1},...,r_{N,1},...,r_{1,t},...,r_{N,t})\in\{0,...,R-1\}^{Nt}$ we define the set 
$$\mathcal Q_{R,\vec r}=[\frac{r_{1,1}}{R},\frac{r_{1,1}+1}{R})\times\cdots\times[\frac{r_{N,1}}{R},\frac{r_{N,1}+1}{R})\times\cdots\times[\frac{r_{1,t}}{R},\frac{r_{1,t}+1}{R})\times\cdots\times[\frac{r_{N,t}}{R},\frac{r_{N,t}+1}{R}).$$ 
Observe that for each $R\in\N$, $\{\mathcal Q_{R,\vec r}\,|\,\vec r\in\{0,...,R-1\}^{Nt}\}$ is a partition of $\mathbb T^{Nt}=[0,1)^{Nt}$.\\
Fix $(\alpha_1,...,\alpha_t)\in\mathfrak M_{t}(\phi_1,...,\phi_N)$. For each $R\in\N$ and each $\vec r\in\{0,...,R-1\}^{Nt}$, let $(n_k^{(R,\vec r)})_{k\in\N}$ be the unique increasing sequence satisfying
$$\{n^{(R,\vec r)}_k\,|\,k\in\N\}=\{n\in\N\,|\,(\phi_1(n)\alpha_1,...,\phi_N(n)\alpha_1,...,\phi_1(n)\alpha_t,...,\phi_N(n)\alpha_t)\text{ mod }1\in\mathcal Q_{R,\vec r}\}.$$
For each $j\in\{1,...,N\}$, let $\phi_j^{(R,\vec r)}:\N\rightarrow\Z$ be defined by
$$\phi_j^{(R,\vec r)}(k)=\phi_j(n^{(R,\vec r)}_k).$$
(Observe that since $\phi_1^{(R,\vec r)}$,...,$\phi_N^{(R,\vec r)}$ are "simultaneous" subsequences of $\phi_1,...,\phi_N$, the sequences $\phi_1^{(R,\vec r)}$,...,$\phi_N^{(R,\vec r)}$ are strongly asymptotically independent.)\\

Let 
$$\mathfrak M'(\phi_1,...,\phi_N,\alpha_1,...,\alpha_t)=\bigcap_{R\in\N}\bigcap_{\vec r\in\{0,...,R-1\}^{Nt}} \mathfrak M_1(\phi_1^{(R,\vec r)},...,\phi_N^{(R,\vec r)}).$$  
Note that, by the inductive hypothesis, $\mathfrak M'(\phi_1,...,\phi_N,\alpha_1,...,\alpha_t)$ has full measure in $\R$. Pick $\alpha \in \mathfrak M'(\phi_1,...,\phi_N,\alpha_1,...,\alpha_t)$. For any $(a_{1,1},...,a_{N,1},....,a_{1,t},...,a_{N,t})\in \Z^{Nt}$ and any $(a_1,...,a_N)\in \Z^N\setminus\{\vec 0\}$,
\begin{multline}\label{4.WeylTechnicalLimit1}
    \lim_{M\rightarrow\infty}\frac{1}{M}\sum_{n=1}^M\exp\left[2\pi i(\sum_{j=1}^N\sum_{s=1}^ta_{j,s}\phi_j(n)\alpha_s+\sum_{j=1}^Na_j\phi_j(n)\alpha)\right]\\
    =\lim_{R\rightarrow\infty}\lim_{M\rightarrow\infty}\frac{1}{M}\sum_{n=1}^{M}\exp\left[2\pi i(\sum_{j=1}^N\sum_{s=1}^ta_{j,s}\phi_j(n)\alpha_s+\sum_{j=1}^Na_j\phi_j(n)\alpha)\right]\\
   =\lim_{R\rightarrow\infty}\lim_{M\rightarrow\infty}\frac{1}{M}\sum_{n=1}^{M}\sum_{\vec r\in\{0,...,R-1\}^{Nt}}\mathbbm 1_{\{n_k^{(R,\vec r)}\,|\,k\in\N\}}(n)\exp\left[2\pi i(\sum_{j=1}^N\sum_{s=1}^ta_{j,s}\phi_j(n)\alpha_s+\sum_{j=1}^Na_j\phi_j(n)\alpha)\right]\\
    =\lim_{R\rightarrow\infty}\sum_{\vec r\in\{0,...,R-1\}^{Nt}}\left(\lim_{M\rightarrow\infty}\frac{1}{M}\sum_{n=1}^{M}\mathbbm 1_{\{n_k^{(R,\vec r)}\,|\,k\in\N\}}(n)\exp\left[2\pi i(\sum_{j=1}^N\sum_{s=1}^ta_{j,s}\phi_j(n)\alpha_s+\sum_{j=1}^Na_j\phi_j(n)\alpha)\right]\right)
    \end{multline}
Fix $R\in\N$ and  $\vec r\in\{0,...,R-1\}^{Nt}$. By our choice of  $(\alpha_1,...,\alpha_t)$ and the definition of $(n_k^{(R,\vec r)})_{k\in\N}$,
$$\lim_{M\rightarrow\infty}\frac{|\{n_k^{(R,\vec r)}\,|\,n_k^{(R,\vec r)}\leq M\}|}{M}=\frac{1}{R^{Nt}}.$$
So,
\begin{multline*}
\lim_{M\rightarrow\infty}\frac{1}{M}\sum_{n=1}^{M}\mathbbm 1_{\{n_k^{(R,\vec r)}\,|\,k\in\N\}}(n)\exp\left[2\pi i(\sum_{j=1}^N\sum_{s=1}^ta_{j,s}\phi_j(n)\alpha_s+\sum_{j=1}^Na_j\phi_j(n)\alpha)\right]\\
    =\lim_{M\rightarrow\infty}\frac{|\{n_k^{(R,\vec r)}\,|\,n_k^{(R,\vec r)}\leq M\}|}{M|\{n_k^{(R,\vec r)}\,|\,n_k^{(R,\vec r)}\leq M\}|}
    \sum_{n=1}^{|\{n_k^{(R,\vec r)}\,|\,n_k^{(R,\vec r)}\leq M\}|}\exp\left[2\pi i(\sum_{j=1}^N\sum_{s=1}^ta_{j,s}\phi^{(R,\vec r)}_j(n)\alpha_s+\sum_{j=1}^Na_j\phi^{(R,\vec r)}_j(n)\alpha)\right]\\
    =\lim_{M\rightarrow\infty}\frac{1}{R^{Nt}|\{n_k^{(R,\vec r)}\,|\,n_k^{(R,\vec r)}\leq M\}|}
    \sum_{n=1}^{|\{n_k^{(R,\vec r)}\,|\,n_k^{(R,\vec r)}\leq M\}|}\exp\left[2\pi i(\sum_{j=1}^N\sum_{s=1}^ta_{j,s}\phi^{(R,\vec r)}_j(n)\alpha_s+\sum_{j=1}^Na_j\phi^{(R,\vec r)}_j(n)\alpha)\right]\\
    = \frac{1}{R^{Nt}}\lim_{M\rightarrow\infty}\left(\frac{1}{M}
    \sum_{n=1}^{M}\exp\left[2\pi i(\sum_{j=1}^N\sum_{s=1}^ta_{j,s}\phi^{(R,\vec r)}_j(n)\alpha_s+\sum_{j=1}^Na_j\phi^{(R,\vec r)}_j(n)\alpha)\right]\right).
\end{multline*}
Observe that for  any $\epsilon>0$ there exists an $R_0\in\N$ such that for any $R\geq R_0$, any  
$\vec r\in\{0,...,R-1\}^{Nt},$
and any $n\in\N$,
$$\left|\exp[2\pi i\sum_{j=1}^N\sum_{s=1}^ta_{j,s}\phi^{(R,\vec r)}_j(n)\alpha_s]-\exp[2\pi i\sum_{j=1}^N\sum_{s=1}^ta_{j,s}\frac{r_{j,s}}{R}]\right|<\epsilon.$$
It follows from \eqref{4.WeylTechnicalLimit1} that
\begin{multline*}
    \lim_{M\rightarrow\infty}\frac{1}{M}\sum_{n=1}^M\exp\left[2\pi i(\sum_{j=1}^N\sum_{s=1}^ta_{j,s}\phi_j(n)\alpha_s+\sum_{j=1}^Na_j\phi_j(n)\alpha)\right]\\
    =\lim_{R\rightarrow\infty}\frac{1}{R^{Nt}}\sum_{\vec r\in\{0,...,R-1\}^{Nt}}\left(\lim_{M\rightarrow\infty}\frac{1}{M}
    \sum_{n=1}^{M}\exp\left[2\pi i(\sum_{j=1}^N\sum_{s=1}^ta_{j,s}\phi^{(R,\vec r)}_j(n)\alpha_s+\sum_{j=1}^Na_j\phi^{(R,\vec r)}_j(n)\alpha)\right]\right)\\
    =\lim_{R\rightarrow\infty}\frac{1}{R^{Nt}}\sum_{\vec r\in\{0,...,R-1\}^{Nt}}\left(\lim_{M\rightarrow\infty}\frac{\exp[2\pi i\sum_{j=1}^N\sum_{s=1}^ta_{j,s}\frac{r_{j,s}}{R}]}{M}
    \sum_{n=1}^{M}\exp[2\pi i\sum_{j=1}^Na_j\phi^{(R,\vec r)}_j(n)\alpha]\right)=0.
\end{multline*}
So $(\alpha_1,...,\alpha_t,\alpha)\in\mathfrak M_{t+1}(\phi_1,...,\phi_N)$.\\
Since $\mathfrak M_t(\phi_1,...,\phi_N)$ has full measure and for any $(\alpha_1,...,\alpha_t)\in \mathfrak M_t(\phi_1,...,\phi_N)$, $\mathfrak M'(\phi_1,...,\phi_N,\alpha_1,...,\alpha_t)$ also has full measure, Fubini's theorem implies that  $\mathfrak M_{t+1}(\phi_1,...,\phi_N)$ has full measure. This completes the induction.\\

To see that for any $(\alpha_1,...,\alpha_t)\in \mathfrak M_t(\phi_1,...,\phi_N)$, $\mathfrak M(\phi_1,...,\phi_N,\alpha_1,...,\alpha_t)$ has full measure, simply note that 
$$\mathfrak M'(\phi_1,...,\phi_N,\alpha_1,...,\alpha_t)\subseteq \mathfrak M(\phi_1,...,\phi_N,\alpha_1,...,\alpha_t).$$
\end{proof}

\begin{thm}\label{4.MainResultExistence}
Let $N\in\N$ and let $\phi_1,...,\phi_N:\N\rightarrow\Z$ be asymptotically independent sequences. Then there exists an increasing sequence $(n_k)_{k\in\N}$ in $\N$ such that for any $\vec \xi=(\xi_1,...,\xi_N)\in\{0,1\}^N$,  there exists a non-trivial weakly mixing Gaussian system  $(\R^\Z,\mathcal A,\gamma_{\vec \xi},T_{\vec \xi})$ with the property that for each $j\in\{1,...,N\}$ and any $A,B\in\mathcal A$, 
\begin{equation}\label{4.MixingOrRigidExpression}
\lim_{k\rightarrow\infty}\gamma_{\vec\xi}(A\cap T_{\vec \xi}^{-\phi_j(n_k)}B)=(1-\xi_j)\gamma_{\vec \xi}(A\cap B)+\xi_j\gamma_{\vec \xi}(A)\gamma_{\vec \xi}(B).
\end{equation}
\end{thm}
\begin{proof}
As in the proof of \cref{3.IndependentPoly}, we will construct spectral measures $\sigma_{\vec \xi}$, $\vec \xi\in\{0,1\}^N$, which have associated Gaussian systems with the desired properties. For each $\vec \xi=(\xi_1,...,\xi_N) \in\{0,1\}^N$, let $\vec b_{\vec \xi}=(b_1^{(\vec \xi)},...,b_N^{(\vec \xi)})\in \Q^N$ be defined as in \eqref{3.SolutionVector}  (so $b^{(\vec \xi)}_j=1-\frac{\xi_j}{2}$ for each $j\in\{1,...,N\}$) and for each $k\in\N$ let
$$\Phi(k)=\max_{j\in\{1,...,N\}}|\phi_j(k)|+1.$$
We claim that there exist  (a) an increasing sequence  $(n_k)_{k\in\N}$ in $\N$  and (b) sequences of irrational numbers $(\alpha_k^{(\vec \xi)})_{k\in\N}$, $\vec \xi\in\{0,1\}^N$, which satisfy:
\begin{enumerate}[(1)]
    \item For each $k\in\N$ and each $\vec \xi\in\{0,1\}^N$, $\alpha_{k}^{(\vec \xi)}\in(0,\frac{1}{2^{k}\Phi(n_{k-1})}]$, where $n_0=1$. So, in particular, 
    $$\lim_{t\rightarrow\infty}\phi_\ell(n_t)\sum_{s=t+1}^\infty |\frac{\alpha_s^{(\vec \xi)}}{2}|=0$$
    for each $\ell\in\{1,...,N\}$.
    \item For each $k\in\N$, each $\vec \xi\in\{0,1\}^N$, and each $\ell\in\{1,...,N\}$,
    $$\|\phi_\ell(n_k)\frac{\alpha_k^{(\vec \xi)}}{2}-\frac{b^{(\vec \xi)}_\ell}{2}\|<\frac{1}{k},$$
    which implies
    $$\lim_{t\rightarrow\infty}\|\phi_\ell(n_t)\frac{\alpha_t^{(\vec \xi)}}{2}- \frac{b^{(\vec \xi)}_\ell}{2}\|=0.$$
    \item For each $k\in\N$, each $\vec \xi\in\{0,1\}^N$, each $\ell\in\{1,...,N\}$, and each $k_0\in \N$ with $k_0<k$,
    $$\|\phi_\ell(n_k)\frac{\alpha_{k_0}^{(\vec \xi)}}{2}\|<\frac{1}{k^2}.$$
    This means that
    $$\lim_{k\rightarrow\infty}\|\phi_\ell(n_k)\frac{\alpha_{k_0}^{(\vec \xi)}}{2}\|=0$$
    fast enough to ensure that
    $$\lim_{k\rightarrow\infty}\sum_{t=1}^k\|\phi_\ell(n_{k+1})\frac{\alpha_t^{(\vec \xi)}}{2}\|=0.$$
\end{enumerate}
Indeed, we define the sequences $(n_k)_{k\in\N}$ and $(\alpha_k^{(\vec \xi)})_{k\in\N}$, $\vec \xi\in\{0,1\}^N$, inductively on $k\in\N$. First, note that there exists an increasing  sequence  $(m_k)_{k\in\N}$ in $\N$ for which the sequences 
$$\psi_j(k)=\phi_j(m_k),\,j\in\{1,...,N\},$$
are strongly asymptotically independent. Let $\vec \xi_1,...,\vec \xi_{2^N}$ be an enumeration of $\{0,1\}^N$. In order to construct the desired sequences, we will need to show that the sequences $(\alpha_k^{(\vec \xi)})_{k\in\N}$, $\vec \xi\in\{0,1\}^N$, satisfy the following additional property:
\begin{enumerate}[(4)]
\item For any $k\in\N$, the sequence 
\begin{multline*}
\big(\phi_1(m_t)\alpha_1^{(\vec \xi_1)},....,\phi_N(m_t)\alpha_1^{(\vec \xi_1)},...,
    \phi_1(m_t)\alpha_{1}^{(\vec \xi_{2^N})},...,,\phi_N(m_t)\alpha_{1}^{(\vec \xi_{2^N})},\\
    \vdots\\
    \phi_1(m_t)\alpha_k^{(\vec \xi_1)},....,\phi_N(m_t)\alpha_k^{(\vec \xi_1)},...,
    \phi_1(m_t)\alpha_{k}^{(\vec \xi_{2^N})},...,,\phi_N(m_t)\alpha_{k}^{(\vec \xi_{2^N})}
    \big),\,t\in\N
\end{multline*}
is uniformly distributed $\mod\,1$.
\end{enumerate}
By \cref{4.WeylsGeneralization}, we can pick $$(\alpha_1^{(\vec \xi_1)},...,,\alpha_{1}^{(\vec \xi_{2^N})})\in (0,\frac{1}{2\Phi(1)}]^{2^N}$$
such that  the sequence 
    $$\big(\phi_1(m_t)\alpha_1^{(\vec \xi_1)},....,\phi_N(m_t)\alpha_1^{(\vec \xi_1)},...,
    \phi_1(m_t)\alpha_{1}^{(\vec \xi_{2^N})},...,\phi_N(m_t)\alpha_{1}^{(\vec \xi_{2^N})}\big),\,t\in\N$$
is uniformly distributed $\mod 1$ (and so, $\alpha_1^{(\vec \xi_1)},...,,\alpha_{1}^{(\vec \xi_{2^N})}$ satisfy (4)). Pick $t_1\in\N$ arbitrarily. Setting $n_1=m_{t_1}$, one can check that $n_1$ and $\alpha_{1}^{(\vec \xi_1)}$,....,$\alpha_{1}^{(\vec \xi_{2^N})}$ 
 satisfy conditions (1), (2), and (3) (note that for $k=1$, (2) is trivial and (3) is vacuous).\\
Fix now $k\in\N$ and  suppose we have chosen $\alpha_{1}^{(\vec \xi)},...,\alpha_{k}^{(\vec \xi)}$, $\vec \xi\in\{0,1\}^N$,  and $n_1<\cdots<n_k$ satisfying conditions (1)-(4). Note that $(0,\frac{1}{2^{k+1}\Phi(n_{k})}]$ has positive measure. By repeatedly applying  \eqref{4.KeyWeylsTypeResult} in  \cref{4.WeylsGeneralization}, we can find $$\alpha_{k+1}^{(\vec \xi_1)},....,\alpha_{k+1}^{(\vec \xi_{2^N})}\in (0,\frac{1}{2^{k+1}\Phi(n_{k})}]$$
such that for each $s\in\{1,...,2^N\}$ the sequence 
\begin{multline*}
\big(\phi_1(m_t)\alpha_1^{(\vec \xi_1)},....,\phi_N(m_t)\alpha_1^{(\vec \xi_1)},...,
    \phi_1(m_t)\alpha_{1}^{(\vec \xi_{2^N})},...,,\phi_N(m_t)\alpha_{1}^{(\vec \xi_{2^N})},\\
    \vdots\\
    \phi_1(m_t)\alpha_k^{(\vec \xi_1)},....,\phi_N(m_t)\alpha_k^{(\vec \xi_1)},...,
    \phi_1(m_t)\alpha_{k}^{(\vec \xi_{2^N})},...,,\phi_N(m_t)\alpha_{k}^{(\vec \xi_{2^N})},\\
    \phi_1(m_t)\alpha_{k+1}^{(\vec \xi_1)},...,\phi_N(m_t)\alpha_{k+1}^{(\vec \xi_{1})},...,\phi_1(m_t)\alpha_{k+1}^{(\vec \xi_s)},...,\phi_N(m_t)\alpha_{k+1}^{(\vec \xi_{s})}
    \big),\,t\in\N
\end{multline*}
is uniformly distributed $\mod 1$. It follows that $\alpha_1^{(\vec \xi)},...,\alpha_{k+1}^{(\vec \xi)}$, $\vec \xi\in\{0,1\}^N$, satisfy (4) and hence one can find $t_{k+1}\in\N$ for which (1)-(3) hold for $n_{k+1}=m_{t_{k+1}}$ and $n_k<n_{k+1}$, completing the induction.\\  

Fix $\vec \xi=(\xi_1,...,\xi_N)\in\{0,1\}^N$. 
By (1)-(3), for any $\epsilon>0$, there exists $k_\epsilon\in\N$ such that for  any $k>k_\epsilon$, any $\ell\in\{1,...,N\}$, and any $\omega\in\{0,1\}^\N$,
\begin{equation}\label{4.TailBound}
\|\phi_\ell(n_k)\sum_{t=k+1}^\infty\frac{\alpha^{(\vec \xi)}_t}{2}\omega(t)\|\leq| \phi_\ell(n_k)\sum_{t=k+1}^\infty\frac{\alpha^{(\vec \xi)}_t}{2}\omega(t)|\leq |\phi_\ell(n_k)|\sum_{t=k+1}^\infty|\frac{\alpha^{(\vec \xi)}_t}{2}|<\epsilon,
\end{equation}
\begin{equation}\label{4.BoundOnCentralTerm}
    \|\phi_\ell(n_k)\frac{\alpha_k^{(\vec \xi)}}{2}\omega(k)-\frac{b^{(\vec \xi)}_\ell}{2}\omega(k)\|\leq \|\phi_\ell(n_k)\frac{\alpha_k^{(\vec \xi)}}{2}-\frac{b^{(\vec \xi)}_\ell}{2}\|<\epsilon,
\end{equation} 
and
\begin{equation}\label{4.TheFirstPart}
 \|\phi_\ell(n_k)\sum_{t=1}^{k-1}\frac{\alpha_t^{(\vec \xi)}}{2}\omega(t)\|
 \leq \sum_{t=1}^{k-1}\|\phi_\ell(n_k)\frac{\alpha_t^{(\vec \xi)}}{2}\|<\epsilon. 
\end{equation}

Let $f:\{0,1\}^\N\rightarrow\R$ be defined by
$$f(\omega)=\sum_{t=1}^\infty \frac{\alpha^{(\vec \xi)}_t}{2}\omega(t).$$
Combining \eqref{4.TailBound}, \eqref{4.BoundOnCentralTerm}, and \eqref{4.TheFirstPart}, one has that for any $\epsilon>0$ there exists a $k_\epsilon\in \N$ such that for any $k>k_\epsilon$, any $\ell\in\{1,...,N\}$, and any $\omega\in\{0,1\}^\N$,
\begin{equation}\label{4.MainEstimate}
\|\phi_\ell(n_k)f(\omega)- \frac{b_\ell^{(\vec \xi)}}{2}\omega(k)\|<\epsilon.
\end{equation}
Let now $f_{\vec \xi}:\{0,1\}^\N\times\{0,1\}^\N\rightarrow \mathbb T$ be defined by
$$f_{\vec \xi}(\omega_1,\omega_2)=2(f(\omega_1)-f(\omega_2))\,\text{mod}\,1=\sum_{t=1}^\infty \alpha_t^{(\vec \xi)}(\omega_1(t)-\omega_2(t))\,\text{mod}\,1.$$
Setting $\sigma_{\vec \xi}=(\mathbb P\times \mathbb P)\circ f_{\vec \xi}^{-1}$ and imitating the proof of \cref{3.IndependentPoly}, we obtain the desired result.
\end{proof}
\begin{cor}
Let $N\in\N$ and let $t\in\{0,...,N\}$. For any $a_1,...,a_N\in\Z$ and any linearly independent polynomials $v_1,...,v_N\in\Z[x]$ with $v_j(0)=0$ for each $j\in\{1,...,N\}$, there exists a non-trivial weakly mixing Gaussian system $(\R^\Z,\mathcal A,\gamma, T)$ and an increasing sequence $(n_k)_{k\in\N}$ in $\N$ such that for any $A,B\in\mathcal A$,
 $$\lim_{k\rightarrow\infty}\gamma(A\cap T^{- v_j(n_k)}B)=\begin{cases}
 \gamma(A\cap T^{-a_j} B)\text{ if }j\leq t,\\
 \gamma(A)\gamma(B)\text{ if }j\in\{1,...,N\}\setminus\{0,...,t\}.
 \end{cases}$$
 \end{cor}
 \begin{proof}
  For each $j\in\{0,...,t\}\setminus\{0\}$ let  $(\phi_j(k))_{k\in\N}=(v_j(k)-a_j)_{k\in\N}$ and for each $j\in\{1,...,N\}\setminus\{0,...,t\}$,  let $(\phi_j(k))_{k\in\N}=(v_j(k))_{k\in\N}$. The result now follows by applying \cref{4.MainResultExistence} to the asymptotically independent sequences $\phi_1,...,\phi_N$. 
 \end{proof}
\section{Interpolating between rigidity and mixing}\label{Section6InterpolationResults}
Our goal in this section is to prove \cref{0.MainInterpolationResult} and obtain \cref{0.IndependentPolysStepinResult} as a corollary. We now restate \cref{0.MainInterpolationResult}. (Recall that we denote by $\N^\N_\infty$ the set of all (strictly) increasing sequences $(n_k)_{k\in\N}$ in $\N$.)
\begin{thm}\label{5.MainInterpolationResult}
Let $N\in\N$, let $\lambda_1,...,\lambda_N\in[0,1]$ and let $\phi_1,...,\phi_N:\N\rightarrow\Z$. Suppose that $\phi_1,...,\phi_N$ satisfy the  following condition: 
\begin{adjustwidth}{0.5cm}{0.5cm}
\underline{Condition C}: There exists an $(n_k)_{k\in\N}\in\N^\N_\infty$ such that for any $\vec \xi=(\xi_1,...,\xi_N)\in\{0,1\}^N$, there exists an aperiodic  $T_{\vec \xi}\in\text{Aut}([0,1],\mathcal B,\mu)$ with the property that for each $j\in\{1,...,N\}$ and any $A,B\in\mathcal B$,
\begin{equation*}
\lim_{k\rightarrow\infty}\mu(A\cap T_{\vec \xi}^{- \phi_j(n_k)}B)=(1-\xi_j)\mu(A\cap B)+\xi_j\mu(A)\mu(B).
\end{equation*}
\end{adjustwidth}
Then the set 
\begin{multline*}
\mathcal O(\phi_1,...,\phi_N)=\{T\in\text{Aut}([0,1],\mathcal B,\mu)\,|\exists (k_\ell)_{\ell\in\N}\in\N^\N_\infty\,\forall j\in\{1,...,N\}\,\forall A,B\in\mathcal B,\\
\lim_{\ell\rightarrow\infty}\mu(A\cap T^{- \phi_j(n_{k_\ell})}B)=(1-\lambda_j)\mu(A\cap B)+\lambda_j\mu(A)\mu(B)\}
\end{multline*}
is a dense $G_\delta$ set.
\end{thm}
Before proving \cref{5.MainInterpolationResult}, we will review the necessary background material on  $\text{Aut}([0,1],\mathcal B,\mu)$.
\subsection{Background on $\text{Aut}([0,1],\mathcal B,\mu)$}\label{Subsection6.1AutBackground}
We will follow the material and the terminology in \cite{halmosBooklectures}. For each $\ell\in\N$, let $E_\ell$ denote the family of the half-open intervals 
$$[\frac{k}{2^\ell},\frac{k+1}{2^\ell}),\,k\in\{0,...,2^\ell-1\}.$$
We call each element of $E_\ell$ a \textbf{dyadic interval of rank $\ell$}. Define the metric $\partial$ on  $\textit{Aut}([0,1],\mathcal B,\mu)$ by
\begin{equation}\label{5.TheIncompleteMetric}
\partial(T,S)=\sum_{\ell\in\N}\frac{1}{2^{2\ell}}\sum_{E \in E_\ell}\mu(T E\triangle S E),
\end{equation}
where $T E\triangle S E$ denotes the symmetric difference between the sets $TE$ and $SE$. The topology induced by  $\partial$ is called the weak topology of $\textit{Aut}([0,1],\mathcal B,\mu)$.\\
With this topology, a sequence $(T_k)_{k\in\N}$ in $\textit{Aut}([0,1],\mathcal B,\mu)$ converges to $T\in\textit{Aut}([0,1],\mathcal B,\mu)$ if and only if $(T_k)_{k\in\N}$ converges to $T$ with respect to the weak operator topology on $L^2(\mu)$ if and only if $(T_k)_{k\in\N}$ converges to $T$ in the strong operator topology on $L^2(\mu)$. Furthermore, $(\textit{Aut}([0,1],\mathcal B,\mu),\partial)$ is a topological group. \\
We remark that while $\textit{Aut}([0,1],\mathcal B,\mu)$ with the weak topology is completely metrisable, the metric space  $(\textit{Aut}([0,1],\mathcal B,\mu),\partial)$ is not complete (i.e. not every Cauchy sequence needs to be convergent).\\

We now turn our attention to some of the dense subsets of $(\textit{Aut}([0,1],\mathcal B,\mu),\partial)$. Given $\ell\in\N$, a transformation $T\in\textit{Aut}([0,1],\mathcal B,\mu)$ is a \textbf{cyclic permutation of the dyadic intervals of rank $\ell$}, if  for any $E\in E_\ell$: (a) $TE\in E_\ell$, (b) there exists an $\alpha\in \R$ such that for any $x\in E$, $Tx=x+\alpha$, and (c) $E_\ell=\{E, TE,...,T^{2^{\ell}-1}E\}$. The following result states  that the cyclic permutations of dyadic intervals are dense in $\textit{Aut}([0,1],\mathcal B,\mu)$ \cite[page 65]{halmosBooklectures}.
\begin{lem}\label{5.WeakAppThm}
Let $T\in \textit{Aut}([0,1],\mathcal B,\mu)$ and let $\epsilon>0$. Then there exists an $\ell_\epsilon\in\N$ such that for any  $\ell>\ell_\epsilon$ there exists a cyclic permutation $S$ of the dyadic intervals of rank $\ell$ such that $\partial (T,S)<\epsilon$. 
\end{lem}
Recall that a transformation  $T\in\textit{Aut}([0,1],\mathcal B,\mu)$ is called aperiodic if the set of $x\in[0,1]$ for which there exists an $n\in\N$ with $T^nx=x$ has measure zero. \cref{5.ConjugacyLemma} below  asserts that the conjugacy class of any aperiodic $T\in\textit{Aut}([0,1],\mathcal B,\mu)$ is dense \cite[page 77]{halmosBooklectures}.
\begin{lem}\label{5.ConjugacyLemma}
Let $T_0\in \textit{Aut}([0,1],\mathcal B,\mu)$ and let $\epsilon>0$. For any aperiodic $T\in\textit{Aut}([0,1],\mathcal B,\mu)$ there exists an $S\in \textit{Aut}([0,1],\mathcal B,\mu)$ such that $\partial (T_0,S^{-1}TS)<\epsilon$.
\end{lem}
\subsection{The proof of \cref{0.IndependentPolysStepinResult} and \cref{0.MainInterpolationResult}}\label{Subsection6.2TheProofOfStepin}
\begin{proof}[Proof of \cref{5.MainInterpolationResult}]
Let the sequence $(n_k)_{k\in\N}$ in $\N$ be as in the statement of \cref{5.MainInterpolationResult}. Recall that for each $\ell\in\N$,  $E_\ell$ denotes the family of all dyadic intervals of rank $\ell$ and let $E(\ell)=\bigcup_{r=1}^\ell E_r$. For each $q,\ell\in\N$, define $\mathcal O(q,\ell)$ to be  the set
\begin{multline*}
\bigcup_{k=\ell}^\infty\bigcap_{E,F\in E(\ell)}\bigcap_{j=1}^N\{T\in \text{Aut}([0,1],\mathcal B,\mu)\,|\,|\mu(E\cap T^{- \phi_j(n_k)}F)-(1-\lambda_j)\mu(E\cap F)-\lambda_j\mu(E)\mu(F)|<\frac{1}{q}\}.
\end{multline*}

Our first claim is that $\mathcal O(\phi_1,...,\phi_N)=\bigcap_{q,\ell\in\N}\mathcal O(q,\ell)$. Clearly, if $T\in\mathcal O(\phi_1,...,\phi_N)$, then $T\in\bigcap_{q,\ell\in\N}\mathcal O(q,\ell)$. Now suppose that $T\in\bigcap_{q,\ell\in\N}\mathcal O(q,\ell)$. It follows that for each $\ell\in\N$ we can find a $k_\ell\geq\ell$ such that 
$$T\in \bigcap_{E,F\in E(\ell)}\bigcap_{j=1}^N\{S\in \text{Aut}([0,1],\mathcal B,\mu)\,|\,|\mu(E\cap S^{- \phi_j(n_{k_\ell})}F)-(1-\lambda_j)\mu(E\cap F)-\lambda_j\mu(E)\mu(F)|<\frac{1}{\ell}\}.$$
By passing to a subsequence, if needed, we can assume that $(k_\ell)_{\ell\in\N}$ is increasing. Furthermore, for any $m\in\N$, any $j\in\{1,...,N\}$, and any  $E,F\in E(m)$,
\begin{equation}\label{5.TIsInO(f_1...,f_N)}
\lim_{\ell\rightarrow\infty}\mu(E\cap T^{- \phi_j(n_{k_\ell})}F)=(1-\lambda_j)\mu(E\cap F)+\lambda_j\mu(E)\mu(F).
\end{equation}
Note that for a fixed $F\in\mathcal B$, the set $\mathcal E_F$ of those $E\in\mathcal B$ for which \eqref{5.TIsInO(f_1...,f_N)} holds is a $\lambda$-system\footnote{
Let $D$ be a family of subsets of a non-empty set  $X$. $D$ is a $\lambda$-system if: (1) $X\in D$, (2) if $A,B\in D$ and  $A\subseteq B$, then $B\setminus A\in D$, and (3) for any collection of sets $\{A_n\,|\,n\in\N\}\subseteq D$ with $A_1\subseteq A_2\subseteq\cdots$, one has $\bigcup_{n\in\N}A_n\in D$.
}
and that for a fixed $E\in\mathcal B$, the set $\Phi_E$ of those $F\in\mathcal B$ for which \eqref{5.TIsInO(f_1...,f_N)} holds is a $\lambda$-system
as well. Also note that  $\bigcup_{\ell\in\N}E_\ell\cup\{\emptyset\}$ is a $\pi$-system\footnote{
Let $P$ be a family of subsets of a non-empty set  $X$. $P$ is a $\pi$-system if $P$ is non-empty and  for any $A,B\in P$,  $A\cap B\in P$.
} with $\bigcup_{\ell\in\N}E_\ell\cup\{\emptyset\}\subseteq \mathcal E_F$ for each $F\in\bigcup_{\ell\in\N}E_\ell$. By  applying the $\pi$-$\lambda$ Theorem (see, for example, \cite[Theorem 2.1.6]{durrett2019probability}) to each
$\mathcal E_F$, $F\in\bigcup_{\ell\in\N}E_\ell$, we see that  \eqref{5.TIsInO(f_1...,f_N)} holds for any $E\in\mathcal B$ and any 
$F\in\bigcup_{\ell\in\N}E_\ell$. Applying the  $\pi$-$\lambda$ Theorem again but now to each $\Phi_E$, $E\in\mathcal B$,  we obtain that  
\eqref{5.TIsInO(f_1...,f_N)} holds 
for arbitrary $E,F\in\mathcal B$  and hence $T\in\mathcal O(\phi_1,...,\phi_N)$.\\

We now show that $\mathcal O(\phi_1,...,\phi_N)$ is $G_\delta$. For any  $E,F\in\bigcup_{\ell\in\N}E_\ell$ define the map $$I_{E,F}:\text{Aut}([0,1],\mathcal B,\mu)\rightarrow [0,1]$$
by $I_{E,F}(T)=\mu(E\cap TF)$.\\
Note that for any given $E,F\in\bigcup_{\ell\in\N}E_\ell$,  $|I_{E,F}(T)-I_{E,F}(S)|\leq \mu(TF\triangle SF)$ and hence $I_{E,F}$ is continuous (with respect to the weak topology). Recall that $\textit{Aut}([0,1],\mathcal B,\mu)$ is a topological group and so, for any $n\in\Z$, the map $T\mapsto T^n$ is continuous. Thus, for each $n\in\Z$ and any $E,F\in\bigcup_{\ell\in\N}E_\ell$, the map $T\mapsto \mu(E\cap T^nF)$ from $\textit{Aut}([0,1],\mathcal B,\mu)$ to $[0,1]$ is continuous as well.  It now follows that for any $q,\ell\in\N$, $\mathcal O(q,\ell)$ is open and hence $\mathcal O(\phi_1,...,\phi_N)$ is $G_\delta$.\\ 

To prove that $\mathcal O(\phi_1,...,\phi_N)$ is dense, it suffices to show that for any $q,\ell\in\N$, any $T_0\in \text{Aut}([0,1],\mathcal B,\mu)$, and any $\epsilon>0$, there exists a $T\in \mathcal O(q,\ell)$ such that $\partial (T_0,T)<\epsilon$. In what follows we will construct a transformation $T\in\text{Aut}([0,1],\mathcal B,\mu)$ with these properties.\\

Fix $q,\ell\in\N$, $T_0\in \text{Aut}([0,1],\mathcal B,\mu)$, and $\epsilon>0$. By \cref{5.WeakAppThm}, there exists a cyclic permutation $R$ of the dyadic intervals of rank $\ell'$ for some $\ell'\geq \ell$ such that 
\begin{equation}\label{5.Cyclification}
    \frac{1}{2^{\ell'}}<\frac{\epsilon}{4}\text{ and }\partial (T_0,R)<\frac{\epsilon}{2}.
\end{equation}
By reindexing $\phi_1,...,\phi_N$, if needed, we assume without loss of generality that 
$$0\leq \lambda_1\leq\lambda_2\leq \cdots\leq \lambda_N\leq 1$$
(we will actually assume that $0<\lambda_1<\cdots<\lambda_N<1$, the general case is handled similarly).\\
By assumption, there exist aperiodic $T_1,...,T_{N+1}\in\text{Aut}([0,1],\mathcal B,\mu)$ such that for each $t\in\{1,...,N+1\}$, each $j\in\{1,...,N\}$ and each $A,B\in\mathcal B$,
$$\lim_{k\rightarrow\infty}\mu(A\cap T_t^{- \phi_j(n_k)}B)=\begin{cases}
\mu(A)\mu(B),\text{ if }j\geq t,\\
\mu(A\cap B),\text{ if }j<t.
\end{cases}$$
By \cref{5.ConjugacyLemma}, we can assume that for each $t\in\{1,...,N+1\}$
\begin{equation}\label{5.T_tandRareClose}
    \partial (R,T_t)<\frac{\epsilon}{4}.
\end{equation}
Furthermore, since the set $\{T_t^n1\,|\,n\in\Z\}$ has measure zero, we assume without loss of generality  that $T_t(1)=1$. Thus, for each $t\in\{1,...,N+1\}$, $T_t([0,1))=[0,1)$. \\

Let $\lambda_0=0$ and $\lambda_{N+1}=1$. For each $t\in\{1,...,N+1\}$, let $\delta_t=\lambda_t-\lambda_{t-1}$ and let 
$$S_t:[0,1)\rightarrow\bigcup_{r=0}^{2^{\ell'}-1}[\frac{r+\lambda_{t-1}}{2^{\ell'}},\frac{r+\lambda_{t}}{2^{\ell'}})$$
be defined by 
$$S_t(x)=\delta_t(x-\frac{r}{2^{\ell'}})+\frac{r+\lambda_{t-1}}{2^{\ell'}}$$
for any $x\in [\frac{r}{2^{\ell'}},\frac{r+1}{2^{\ell'}})$. We remark that $S_t$ is a bijection and both $S_t$ and $S^{-1}_t$ are measurable.\\
We now define $T:[0,1]\rightarrow[0,1]$ by
$$T(x)=\begin{cases}
S_t\circ T_t\circ S_t^{-1}(x),\text{ if }\exists t\in\{1,...,N+1\},\,x\in\bigcup_{r=0}^{2^{\ell'}-1}[\frac{r+\lambda_{t-1}}{2^{\ell'}},\frac{r+\lambda_{t}}{2^{\ell'}})=S_t([0,1)),\\
1,\,\text{if }x=1.
\end{cases}$$
It now remains to show that (i) $T\in\text{Aut}([0,1],\mathcal B,\mu)$, (ii) $T\in\mathcal O(q,\ell)$, and (iii) $\partial (T_0,T)<\epsilon$.\\

\underline{(i)}: We will now show that $T\in\text{Aut}([0,1],\mathcal B,\mu)$. For each $t\in\{1,...,N+1\}$, 
\begin{equation}\label{5.DefnMapOnFiber}
S_t\circ T_t\circ S_t^{-1}:\bigcup_{r=0}^{2^{\ell'}-1}[\frac{r+\lambda_{t-1}}{2^{\ell'}},\frac{r+\lambda_{t}}{2^{\ell'}})\rightarrow\bigcup_{r=0}^{2^{\ell'}-1}[\frac{r+\lambda_{t-1}}{2^{\ell'}},\frac{r+\lambda_{t}}{2^{\ell'}})
\end{equation}
is an invertible measurable function with measurable inverse $S_t\circ T_t^{-1}\circ S_t^{-1}$. Note that for any measurable $A\subseteq [0,1)$, $\mu(S_t(A))=\delta_t\mu(A)$ and, consequently, for any measurable $A\subseteq S_t([0,1))$,
$\mu(S_t^{-1}A)=\frac{1}{\delta_t}\mu(A)$. It follows that for any measurable $A\subseteq S_t([0,1))$,
\begin{equation}\label{5.MeasurePreservingFiber}
\mu(S_t\circ T_t\circ S_t^{-1}(A))=\delta_t\mu(T_t\circ S_t^{-1}(A))
=\delta_t\mu( S_t^{-1}(A))=\delta_t\cdot\frac{1}{\delta_t}\mu(A)=\mu(A)
\end{equation}
and similarly $\mu(S_t\circ T_t^{-1}\circ S_t^{-1}(A))=\mu(A)$.\\
Let $A\subseteq [0,1]$ be measurable and for each $t\in\{1,...,N+1\}$, let $A_t=A\cap S_t([0,1))$. Since $A=\bigcup_{t=1}^{N+1}A_t$ up to a set of measure zero and $A_1,...,A_{N+1}$ are disjoint, \eqref{5.MeasurePreservingFiber} implies
$$\mu(TA)=\mu(\bigcup_{t=1}^{N+1}TA_t)=\sum_{t=1}^{N+1}\mu(S_t\circ T_t\circ S_t^{-1}(A_t))=\sum_{t=1}^{N+1}\mu(A_t)=\mu(\bigcup_{t=1}^{N+1}A_t)=\mu(A)$$
and $\mu(T^{-1}A)=\mu(A)$. Thus, $T\in\text{Aut}([0,1],\mathcal B,\mu)$.\\

\underline{(ii)}: To prove that $T\in\mathcal O(q,\ell)$, we will first note that for each $E\in E(\ell')$ and each $t\in\{1,...,N+1\}$, $E\cap S_t([0,1))=S_t(E)$. We also note that, by \eqref{5.DefnMapOnFiber}, for any $t\in\{1,...,N+1\}$, $T\big(S_t([0,1))\big)=S_t([0,1))$. Thus, for any $j\in\{1,...,N\}$ and any $E,F\in E(\ell)\subseteq  E(\ell')$,
\begin{multline*}
\lim_{k\rightarrow\infty}\mu(E\cap T^{- \phi_j(n_k)}F)=\lim_{k\rightarrow\infty}\sum_{t=1}^{N+1}\mu(E\cap T^{- \phi_j(n_k)}F\cap S_t([0,1)))\\
=\lim_{k\rightarrow\infty}\sum_{t=1}^{N+1}\mu([E\cap S_t([0,1))]\cap T^{- \phi_j(n_k)}[F\cap S_t([0,1))])=\lim_{k\rightarrow\infty}\sum_{t=1}^{N+1}\mu[S_t(E)\cap T^{- \phi_j(n_k)}(S_tF)]\\
=\lim_{k\rightarrow\infty}\sum_{t=1}^{N+1}\mu[S_t(E)\cap(S_t\circ T_t^{- \phi_j(n_k)}\circ S_t^{-1})(S_tF)]
=\lim_{k\rightarrow\infty}\sum_{t=1}^{N+1}\mu[S_t(E\cap T_t^{- \phi_j(n_k)}F)]\\
=\lim_{k\rightarrow\infty}\sum_{t=1}^{N+1}\delta_t\mu(E\cap T_t^{- \phi_j(n_k)}F)=\sum_{t=j+1}^{N+1}\delta_t\mu(E\cap F)+\sum_{t=1}^j\delta_t\mu(E)\mu(F)\\
=\sum_{t=j+1}^{N+1}(\lambda_t-\lambda_{t-1})\mu(E\cap F)+\sum_{t=1}^j(\lambda_t-\lambda_{t-1})\mu(E)\mu(F)=(1-\lambda_j)\mu(E\cap F)+\lambda_j\mu(E)\mu(F).
\end{multline*}
So $T\in\mathcal O(q,\ell)$.\\

\underline{(iii)}: By \eqref{5.Cyclification}, to prove that $\partial (T_0,T)<\epsilon$, all we need to show is that $\partial (R,T)<\frac{\epsilon}{2}$. Note that for any $E\in E_{\ell'}$, $RE\in E_{\ell'}$. So, for any $E\in E(\ell')$ and any $t\in\{1,...,N+1\}$, $RE\cap S_t([0,1))=S_t(R(E))$.  It follows that 
\begin{multline*}
    \sum_{\ell=1}^{\ell'}\frac{1}{2^{2\ell}}\sum_{E\in E_\ell} \mu(RE\triangle TE)=\sum_{\ell=1}^{\ell'}\frac{1}{2^{2\ell}}\sum_{E\in E_\ell}\sum_{t=1}^{N+1} \mu((RE\triangle TE)\cap S_t([0,1)))\\
    =\sum_{\ell=1}^{\ell'}\frac{1}{2^{2\ell}}\sum_{E\in E_\ell}\sum_{t=1}^{N+1} \mu([RE\cap S_t([0,1))]\triangle [TE\cap S_t([0,1))])=\sum_{\ell=1}^{\ell'}\frac{1}{2^{2\ell}}\sum_{E\in E_\ell}\sum_{t=1}^{N+1} \mu([S_tRE]\triangle [TS_tE])\\
    =\sum_{\ell=1}^{\ell'}\frac{1}{2^{2\ell}}\sum_{E\in E_\ell}\sum_{t=1}^{N+1} \mu([S_tRE]\triangle (S_t\circ T_t \circ S
    ^{-1}_t)(S_tE))= \sum_{\ell=1}^{\ell'}\frac{1}{2^{2\ell}}\sum_{E\in E_\ell}\sum_{t=1}^{N+1} \mu(S_t(RE\triangle T_t E))\\
    =\sum_{\ell=1}^{\ell'}\frac{1}{2^{2\ell}}\sum_{E\in E_\ell}\sum_{t=1}^{N+1}\delta_t\mu(RE\triangle T_tE)=\sum_{t=1}^{N+1}\delta_t\sum_{\ell=1}^{\ell'}\frac{1}{2^{2\ell}}\sum_{E\in E_\ell}\mu(RE\triangle T_tE)\leq \sum_{t=1}^{N+1}\delta_t\partial(R,T_t).
\end{multline*}
By \eqref{5.T_tandRareClose}, $\partial(R,T_t)<\frac{\epsilon}{4}$, so
$$ \sum_{\ell=1}^{\ell'}\frac{1}{2^{2\ell}}\sum_{E\in E_\ell} \mu(RE\triangle TE)\leq \sum_{t=1}^{N+1}\delta_t\partial(R,T_t)<\frac{\epsilon}{4}.$$
Finally, since by our choice of $\ell'$, $\frac{1}{2^{\ell'}}<\frac{\epsilon}{4}$, we obtain
\begin{multline*}
    \partial (R,T)=\sum_{\ell\in\N}\frac{1}{2^{2\ell}}\sum_{E\in E_\ell}\mu(RE\triangle TE)\\\
    =\sum_{\ell=1}^{\ell'}\frac{1}{2^{2\ell}}\sum_{E\in E_\ell} \mu(RE\triangle TE)+\sum_{\ell=\ell'+1}^\infty \frac{1}{2^{2\ell}}\sum_{E\in E_\ell}\mu(RE\triangle TE)\\
    \leq\sum_{\ell=1}^{\ell'}\frac{1}{2^{2\ell}}\sum_{E\in E_\ell} \mu(RE\triangle TE)+\frac{1}{2^{\ell'}}<\frac{\epsilon}{4}+\frac{\epsilon}{4}=\frac{\epsilon}{2}.
\end{multline*}
We are done.
\end{proof}
We now obtain \cref{0.IndependentPolysStepinResult} as a corollary of \cref{4.MainResultExistence} and \cref{5.MainInterpolationResult}.
\begin{thm}\label{6.IndependentPolysStepinResult}
Let $N\in\N$ and let $\lambda_1,...,\lambda_N\in [0,1]$.  For any asymptotically independent sequences $\phi_1,...,\phi_N:\N\rightarrow\Z$, the set
\begin{multline*}
    \mathcal O=\{T\in\text{Aut}([0,1],\mathcal B,\mu)\,|\,\exists (n_k)_{k\in\N}\in\N^\N_\infty\,\forall j\in\{1,...,N\}\,\forall A,B\in\mathcal B,\\
    \lim_{k\rightarrow\infty}\mu(A\cap T^{- \phi_j(n_k)}B)=(1-\lambda_j)\mu(A\cap B)+\lambda_j\mu(A)\mu(B)\}
\end{multline*}
is a dense $G_\delta$ set.
\end{thm}
\begin{proof}
 By an argument similar to the one used in the proof of \cref{5.MainInterpolationResult}, $\mathcal O$ is a $G_\delta$ set. Combining  \cref{4.MainResultExistence} and \cref{5.MainInterpolationResult}, we see that $\mathcal O$ contains a dense $G_\delta$ set. Hence, it is a dense $G_\delta$ set.
\end{proof}
\section{Families of non-asymptotically independent sequences for which Condition C holds}\label{Section7OtherFamilies}
In this section we will show that, as mentioned in the Introduction, Condition C in \cref{0.MainInterpolationResult} is satisfied by families of sequences which are not asymptotically independent. The following result, which also follows from  \cite[Theorem 3.11]{BKLUltrafilterPoly}, provides some examples of such families of sequences. (Our proof is different from that of \cite[Theorem 3.11]{BKLUltrafilterPoly}.)
\begin{thm}
Let $N\geq 2$, let $A_1,...,A_{2^N-2}$ be an enumeration of the non-empty proper subsets of $\{1,...,N\}$, and let $p_1,...,p_{2^N-2}\in\N$ be distinct prime numbers. For each $j\in\{1,...,N\}$, set
\begin{equation}\label{A.Defn_q_j}
q_j=\prod_{\{n\in\{1,...,2^N-2\}\,|\,j\in A_n\}}p_n.    
\end{equation}
and put $\phi_j(k)=q_jk$, $k\in\N$. Then there exists an increasing sequence $(n_k)_{k\in\N}$ in $\N$ such that for any $\vec \xi=(\xi_1,...,\xi_N)\in\{0,1\}^N$, there exists a non-trivial weakly mixing Gaussian system  $(\R^\Z,\mathcal A,\gamma_{\vec \xi},T_{\vec \xi})$ with the property that for each $j\in\{1,...,N\}$ and any $A,B\in\mathcal A$,
\begin{equation}\label{A.TheLimit}
\lim_{k\rightarrow\infty}\gamma_{\vec\xi}(A\cap T_{\vec \xi}^{- \phi_j(n_k) }B)=(1-\xi_j)\gamma_{\vec \xi}(A\cap B)+\xi_j\gamma_{\vec \xi}(A)\gamma_{\vec \xi}(B).
\end{equation}
\end{thm}
\begin{proof}
 Let $A_0=\{1,...,N\}$, let $A_{2^N-1}=\emptyset$, and let $M$ be the least prime number with the property that for each $n\in\{1,...,2^N-2\}$, $M>p_n$. Put $p_{2^N-1}=M$ and define the sequence $(n_k)_{k\in\N}$  by 
 $$n_k=(\prod_{r=1}^{2^N-1}p_r)^{2k}k!,\,k\in\N.$$
For each $n\in\{0,...,2^N-1\}$, define $\vec \xi_n=(\xi_1^{(n)},...,\xi_N^{(n)})\in\{0,1\}^N$ by
$$\xi_j^{(n)}=1-\mathbbm 1_{A_n}(j),\, j\in\{1,...,N\}.$$
(Observe that $\{\vec \xi_n\,|\,n\in\{0,...,2^N-1\}\}=\{0,1\}^N$.)\\

Fix $n\in\{0,1,...,2^N-1\}$, put $p_0=1$, and let $c_n=\max\{p_n-1,1\}$. Consider the product space
$$X_n=\{0,...,c_n\}^\N$$
and let $\mathbb P_n$ be the Borel probability measure on $X_n$ defined by the infinite product of the  normalized counting measure on $\{0,...,c_n\}$. Let $f_n:X_n\times X_n\rightarrow \mathbb T$ be defined by 
$$f_n(\omega_1,\omega_2)=\sum_{t=1}^\infty\frac{1}{n_t}\frac{\omega_1(t)-\omega_2(t)}{p_n}\,\text{mod}\,1.$$
Clearly $f_n$ is continuous.\\
Set the probability measure $\sigma_{\vec\xi_n}$ on $\mathbb T$ to equal $(\mathbbm P_n\times\mathbb P_n)\circ f_n^{-1}$. Note that for each $k\in\Z$,
\begin{multline*}
\int_\mathbb T e^{2\pi ikx}\text{d}\sigma_{\vec \xi_n}(x)=\int_{X_n}\int_{X_n}e^{2\pi i k(\sum_{t=1}^\infty\frac{1}{n_t}\frac{\omega_1(t)-\omega_2(t)}{p_n})}\text{d}\mathbb P_n(\omega_1)\text{d}\mathbb P_n(\omega_2)\\
=|\int_{X_n}e^{2\pi ik(\sum_{t=1}^\infty\frac{1}{n_t}\frac{\omega(t)}{p_n})}\text{d}\mathbb P_n(\omega)|^2.
\end{multline*}
It follows that $\sigma_{\vec \xi_n}$ is a (non-zero, positive) symmetric probability measure. We claim that the non-trivial Gaussian system $(\R^\Z,\mathcal A,\gamma_{\vec \xi_n}, T_{\vec\xi_n})$ associated with $\sigma_{\vec\xi_n}$ is weakly mixing and satisfies \eqref{A.TheLimit}.
By \cref{1.TheWienerConsequence} and \cref{1.NotionsOfMixing}, it suffices to show that (i) $\sigma_{\vec\xi_n}$ is continuous and (ii) that for each $j\in\{1,...,N\}$ and any $m\in\Z$, 
\begin{equation}\label{A.TheSpectralLimit}
\lim_{k\rightarrow\infty}\int_\mathbb T e^{2\pi i(\phi_j(n_k)+m)x}\text{d}\sigma_{\vec\xi_n}(x)=(1-\xi_j^{(n)})\int_\mathbb T e^{2\pi imx}\text{d}\sigma_{\vec\xi_n}(x).
\end{equation}

\underline{(i)}: We will now show that $\sigma_{\vec\xi_n}$ is continuous. For this, let $j\in\{1,...,N\}$ and note that
\begin{multline}\label{A.UniformLimit.1}
\lim_{k\rightarrow\infty}\|\phi_j(n_k)\sum_{t=1}^\infty\frac{1}{n_t}\frac{\omega(t)}{p_nM}-\frac{q_j\omega(k)}{p_nM}\|=\lim_{k\rightarrow\infty}\|\sum_{t=1}^\infty\frac{n_k}{n_t}\frac{q_j\omega(t)}{p_nM}-\frac{q_j\omega(k)}{p_nM}\|\\
=\lim_{k\rightarrow\infty}\|\sum_{t=1}^\infty\frac{k!}{t!}\frac{(\prod_{r=1}^{2^N-1}p_r)^{2(k-t)}q_j\omega(t)}{p_nM}-\frac{q_j\omega(k)}{p_nM}\|=0
\end{multline}
uniformly in $\omega\in X_n$. Thus,
\begin{equation}\label{A.UniformLimit}
\lim_{k\rightarrow\infty}\|\phi_j(n_k)\sum_{t=1}^\infty\frac{1}{n_t}\frac{\omega_1(t)}{p_nM}-\phi_j(n_k)\sum_{t=1}^\infty\frac{1}{n_t}\frac{\omega_2(t)}{p_nM}\|-\|\frac{q_j\omega_1(k)}{p_nM}-\frac{q_j\omega_2(k)}{p_nM}\|=0
\end{equation}
uniformly in $(\omega_1,\omega_2)\in X_n\times X_n$.\\
By \eqref{A.Defn_q_j},  $q_j$ and $M$ are relatively prime and hence for any $a,b\in\{0,...,c_n\}$,
\begin{equation}\label{A.PrimeSubgroupsCongruenceMode1}
\frac{q_ja}{p_nM}\equiv\frac{q_jb}{p_nM}\,\text{mod}\,1\text{ if and only if $a=b$.}
\end{equation}
The continuity of $\sigma_{\vec \xi_n}$ now follows from \eqref{A.UniformLimit} and \eqref{A.PrimeSubgroupsCongruenceMode1} by noting that $\mathbb P_n$ is an atomless measure and arguing as in the proof of \cref{3.IndependentPoly}.\\

\underline{(ii)}: By \eqref{A.UniformLimit.1}, for any $j\in\{1,...,N\}$, 
\begin{equation*}
\lim_{k\rightarrow\infty}\|\phi_j(n_k)\sum_{t=1}^\infty\frac{1}{n_t}\frac{\omega(t)}{p_n}-\frac{q_j\omega(k)}{p_n}\|
=
\lim_{k\rightarrow\infty}|M|\|\phi_j(n_k)\sum_{t=1}^\infty\frac{1}{n_t}\frac{\omega(t)}{p_nM}-\frac{q_j(\omega(k)}{p_nM}\|=0
\end{equation*}
uniformly on $\omega\in X_n$. So for each $m\in\Z$,
\begin{multline*}
\lim_{k\rightarrow\infty}\int_\mathbb T e^{2\pi i (\phi_j(n_k)+m)x}\text{d}\sigma_{\vec \xi_n}(x)\\
=\lim_{k\rightarrow\infty} |\int_{X_n}e^{2\pi i(\phi_j(n_k)+m)(\sum_{t=1}^\infty\frac{1}{n_t}\frac{\omega(t)}{p_n})}\text{d}\mathbb P_n(\omega)|^2
=\lim_{k\rightarrow\infty} |\int_{X_n}e^{2\pi i\frac{q_j\omega(k)}{p_n}}e^{2\pi i m(\sum_{t=1}^\infty\frac{1}{n_t}\frac{\omega(t)}{p_n})}\text{d}\mathbb P_n(\omega)|^2\\
=|\frac{1}{c_n+1}\sum_{r=0}^{c_n}e^{2\pi i\frac{q_jr}{p_n}}|^2|\int_{X_n}e^{2\pi i m(\sum_{t=1}^\infty\frac{1}{n_t}\frac{\omega(t)}{p_n})}\text{d}\mathbb P_n(\omega)|^2
=|\frac{1}{c_n+1}\sum_{r=0}^{c_n}e^{2\pi i\frac{q_jr}{p_n}}|^2\int_\mathbb T e^{2\pi imx}\text{d}\sigma_{\vec \xi_n}(x).
\end{multline*}
By \eqref{A.Defn_q_j}, for each $j\in\{1,...,N\}$,
$$|\frac{1}{c_n+1}\sum_{r=0}^{c_n}e^{2\pi i\frac{q_jr}{p_n}}|^2=\mathbbm 1_{A_n}(j)=1-\xi_j^{(n)},$$
which implies that \eqref{A.TheSpectralLimit} holds.
\end{proof}

\textbf{Acknowledgments:} The author would like to thank Professor Vitaly Bergelson for the  question which motivated this paper and his valuable input during the preparation of the manuscript. The author also thanks the anonymous  referee to whom the much clearer and more succinct presentation of the material in Section 2 is owed.

\bibliographystyle{abbrvnat}
\bibliography{Bib.bib}

\end{document}